\documentclass[a4paper,11pt,reqno]{amsart}
\usepackage[utf8]{inputenc}
\usepackage{amsmath}
\usepackage{amssymb}
\usepackage{amsthm}
\usepackage{hyperref}
\usepackage{color}
\usepackage{a4wide}
\usepackage{hieroglf}
\usepackage{ marvosym }
\usepackage{comment}
\usepackage{wasysym}

\newcommand{\fraka}{\mathfrak{a}}
\newcommand{\frakg}{\mathfrak{g}}
\newcommand{\frakh}{\mathfrak{h}}
\newcommand{\frakk}{\mathfrak{k}}

\newcommand{\frakq}{\mathfrak{q}}

\newcommand{\CC}{\mathbb{C}}

\newcommand{\NN}{\mathbb{N}}

\newcommand{\RR}{\mathbb{R}}
\newcommand{\ZZ}{\mathbb{Z}}

\newcommand{\calH}{\mathcal{H}}
\newcommand{\calL}{\mathcal{L}}

\newcommand{\0}{{\bf 0}}
\newcommand{\1}{{\bf 1}}

\DeclareMathOperator{\GL}{GL}
\DeclareMathOperator{\upO}{O}

\DeclareMathOperator{\SL}{SL}

\DeclareMathOperator{\Ind}{Ind}
\DeclareMathOperator{\Res}{Res}
\DeclareMathOperator{\tr}{tr}

\DeclareMathOperator{\Ad}{Ad}

\DeclareMathOperator{\proj}{proj}

\DeclareMathOperator{\id}{id}
\DeclareMathOperator{\sgn}{sgn}

\DeclareMathOperator{\diag}{diag}

\renewcommand\Re{\operatorname{Re}}

\newcommand{\hol}{\textup{hol}}
\newcommand{\ahol}{\textup{ahol}}

\DeclareMathOperator{\fraksl}{\mathfrak{sl}}

\theoremstyle{plain}
\newtheorem{theorem}{Theorem}[section]
\newtheorem{proposition}[theorem]{Proposition}
\newtheorem{lemma}[theorem]{Lemma}
\newtheorem{corollary}[theorem]{Corollary}

\newtheorem{Remark}[theorem]{Remark}

\DeclareMathOperator{\ankh}{\text{\Large \Ankh}}
\DeclareMathOperator{\caprisun}{\text{\Capricorn}}

\theoremstyle{definition}

\usepackage{anysize,eso-pic,graphicx}

\title[An explicit Plancherel formula for the hyperboloid]{An explicit Plancherel formula for line bundles over the one-sheeted hyperboloid}
\author{Frederik Bang--Jensen \& Jonathan Ditlevsen}
\begin{document}

\maketitle
\begin{abstract}
In this paper we consider $G=\SL(2,\RR)$ and $H$ the subgroup of diagonal matrices. Then $X=G/H$ is a unimodular homogeneous space which can be identified with the one-sheeted hyperboloid. For each unitary character $\chi$ of $H$ we decompose the induced representations $\Ind_H^G(\chi)$ into irreducible unitary representations, known as a Plancherel formula. This is done by studying explicit intertwining operators between $\Ind_H^G(\chi)$ and principal series representations of $G$. These operators depends holomorphically on the induction parameters.

\end{abstract}

\section*{Introduction}

\noindent The Plancherel formula for a unimodular homogeneous space $X=G/H$ of a Lie group $G$ describes the decomposition of the left-regular representation of $G$ on $L^2(X)$ into irreducible unitary representations. More generally, one can ask for the decomposition of $L^2(G\times_H V_\chi)$, the $L^2$-sections of a homogeneous vector bundle associated with a unitary representation $(\chi,V_\chi)$ of $H$. In representation theoretic language, this corresponds to the induced representation $\Ind_H^G(\chi)$ of $G$, and for the trivial representation $\chi=\1$ we recover $L^2(G/H)$.

By abstract theory, the unitary representation $\Ind_H^G(\chi)$ decomposes into a direct integral of irreducible unitary representations of $G$, i.e. there exists a measure $\mu$ on the unitary dual $\widehat{G}$ of $G$ and a multiplicity function $m:\widehat{G}\to\NN\cup\{\infty\}$ such that
$$ \Ind_H^G(\chi) \simeq \int^\oplus_{\widehat{G}}m(\pi)\cdot\pi\,d\mu(\pi). $$
An \emph{abstract} Plancherel formula describes the support of the Plancherel measure $\mu$ as well as the multiplicity function $m$. Such abstract Plancherel formulas have been established for certain classes of homogeneous spaces such as semisimple symmetric spaces (see e.g. \cite{vdB}).

However, for some applications an abstract Plancherel formula is not sufficient, and a more \emph{explicit} version is needed (see e.g. \cite{FW,Wei}). By this, we mean an explicit formula for the measure $\mu$ as well as explicit linearly independent intertwining operators $A_{\pi,j}:\Ind_H^G(\chi)^\infty\to\pi^\infty$, $j=1,\ldots,m(\pi)$, for $\mu$-almost every $\pi\in\widehat{G}$ such that
$$ \|f\|_{\Ind_H^G(\chi)}^2 = \int_{\widehat{G}} \sum_{j=1}^{m(\pi)}\|A_{\pi,j}f\|_\pi^2\,d\mu(\pi) \qquad (f\in\Ind_H^G(\chi)^\infty). $$

Such an explicit Plancherel formula is for instance known for Riemannian symmetric spaces $X=G/K$, where the Plancherel measure $\mu$ is explicitly given in terms of Harish-Chandra's $c$-function and the intertwining operators $A_{\pi,j}$ can be described in terms of spherical functions (see e.g. \cite{Hel}, and also \cite{Shi} for the case of line bundles over Hermitian symmetric spaces). This explicit Plancherel formula has recently been applied in the context of branching problems for unitary representations where the explicit Plancherel measure and in particular its singularities play a crucial role (see e.g. \cite{FW,Wei}). In order to apply the same strategy to other branching problems, explicit Plancherel formulas are needed for more general homogeneous spaces.

In this paper, we determine the explicit Plancherel formula for line bundles over the one-sheeted hyperboloid $X=G/H$, where $G=\SL(2,\RR)$ and $H$ the subgroup of diagonal matrices. This specific Plancherel formula has direct applications to branching problems for the pairs $(\SL(2,\RR)\times \SL(2,\RR),\diag(\SL(2,\RR))$ and $(\GL(3,\RR),\GL(2,\RR))$. The homogeneous Hermitian line bundles over $X$ are parameterized by $\varepsilon\in\ZZ/2\ZZ$ and $\lambda\in i\RR$, the corresponding unitary character of $H$ being
$$ \chi_{\varepsilon,\lambda}\begin{pmatrix}t&0\\0&t^{-1}\end{pmatrix} = \sgn(t)^\varepsilon|t|^\lambda \qquad (t\in\RR^\times). $$
We find intertwining operators $A_{\lambda,\mu}^\xi: \Ind_H^G(\chi_{\lambda,\varepsilon})\to \Ind_P^G(\varepsilon\otimes e^\mu\otimes 1),$ $\xi=0,1$ 
between the line bundles over $X$ and the principal series representation (see Proposition \ref{definition af intertwiner}).
\begin{theorem}[See Theorem \ref{hovedresultat}]
For $f\in \Ind_H^G(\chi_{\lambda,\varepsilon})$, $\lambda \in i\RR$ and $\varepsilon\in \{0,1\}$ we have 
\begin{equation}\label{intro hovedresultat}
\| f \|^2=\int_{i\RR} \sum_{\xi=0}^1 \| \mathbf{A}_{\lambda,\mu}^\xi f \|^2 \frac{d\mu}{|a(\mu,\varepsilon)|^2}+\sum_{\mu \in 1-\varepsilon-2\NN} c(\mu,\varepsilon) \| \mathbb{A}_{\lambda,\mu}^\varepsilon f \|^2,
\end{equation}
where $\mathbf{A}_{\lambda,\mu}$ and $\mathbb{A}_{\lambda,\mu}$ are some combinations of $A_{\lambda,\mu}^0$ and $A_{\lambda,\mu}^1$. 
\end{theorem}
\noindent The proof of \eqref{intro hovedresultat} consists of two steps. First, we prove \eqref{intro hovedresultat} in the case $\lambda=0$ separately for each $K$-isotypic component. On a fixed $K$-isotypic component, the intertwining operators $A_{\lambda,\mu}^\xi$ are essentially Fourier--Jacobi transforms, and the Plancherel formula follows from the spectral decomposition of the corresponding ordinary second order differential operator by Sturm--Liouville theory. The main difficulty is that the continuous spectrum occurs with multiplicity two, while the discrete part occurs with multiplicity-one, and it is non-trivial to find the right linear combination of $A_\mu^0$ and $A_\mu^1$ that corresponds to a direct summand. In fact, this linear combination is very different for the cases $\varepsilon=0$ and $\varepsilon=1$. In the second step, we show that, as a representation of $G$, $L^2(G/H,\calL_{\varepsilon,\lambda})$ is independent of $\lambda$, and by finding an explicit unitary isomorphism $L^2(G/H,\calL_{\varepsilon,\lambda})\to L^2(G/H,\calL_{\varepsilon,0})$ we deduce the claimed formula.

We remark that for $\varepsilon=0$ and general $\lambda\in i\RR$ the Plancherel formula was recently obtained by Zhu~\cite{Zhu}. Moreover, for $\varepsilon=0$ and $\lambda=0$ our Plancherel formula can be viewed as a special case of the one for pseudo-Riemannian real hyperbolic spaces $\upO(p,q)/\upO(p,q-1)$ with $p=1$ and $q=2$ which was obtained by Faraut~\cite{Far}, Rossmann~\cite{Ros} and Strichartz~\cite{Str}. Note also that the corresponding \emph{abstract} Plancherel formula, i.e. the description of the representations occurring in the direct integral decomposition, also follows from the general theory (see e.g. \cite{vdB}).

\subsection*{Acknowledgements}
We would like to thank our supervisor Jan Frahm for his help and input on the topics of this paper.
\subsection*{Notation}
$\NN =\{1,2,3\dots\}$, $ \NN_0=\NN\cup \{0\}$. For $A\subseteq \RR$ and $ b,c\in \RR$ we denote by $b+cA=\{b+ca\,|\,a\in A\}$. The Pochhammer symbol is $(x)_n=x(x+1)\cdots(x+n-1)$. We denote Lie groups
by Roman capitals and their corresponding Lie algebras by the corresponding Fraktur
lower cases. For $m\in \ZZ$ we let $[m]_2\in \{0,1\} $ be the remainder of $m$ after division by 2.

\section{The principal series of \texorpdfstring{$\SL(2,\RR)$}{SL(2,R)}}\label{Principal series section}
\noindent In this section we recall some results about the representation theory of $\SL(2,\RR)$ following \cite{C20}. Let $G=\SL(2,\RR)$ and consider the following subgroups
$$M=\{\pm I\},\quad A=\left \{ \begin{pmatrix}t & 0 \\ 0 & t^{-1}\end{pmatrix}\, :\, t\in \RR_{>0}\right \},\quad N=\left \{ \begin{pmatrix}1 & x \\ 0 & 1 \end{pmatrix}\,:\, x\in \RR \right \},$$
then $P=MAN$ is a minimal parabolic subgroup of $G$. Identify $\widehat{M}\cong \ZZ/2\ZZ$ by mapping $\varepsilon\in \ZZ/2\ZZ$ to the character
\[
M\to \{\pm 1\},\,\quad \begin{pmatrix}
\pm 1 & 0\\
0 & \pm 1
\end{pmatrix}\mapsto (\pm 1)^\varepsilon. 
\]
Further, we identify $\fraka_\CC^*\cong \CC$ by mapping $\lambda\mapsto \lambda\big(\diag(1,-1)\big).$ We can then observe that any character of $H:=MA$ is of the form $\chi_{\varepsilon,\lambda}=\varepsilon\otimes e^\lambda$ where 
\[
\chi_{\varepsilon,\lambda}
\begin{pmatrix} t & 0 \\
0 & t^{-1} \end{pmatrix}
= |t|^\lambda_\varepsilon:=\sgn(t)^\varepsilon|t|^\lambda,\quad \quad (t\in \RR^\times).
\]
As the commutator subgroup of $P$ is $N$ the characters of $P$ is of the form $\varepsilon\otimes e^\mu\otimes 1$ and these characters are unitary exactly when $\lambda\in i\RR$.

Let $\varepsilon\in \ZZ/2\ZZ$ and $\mu \in \CC$. For any character $\varepsilon\otimes e^\mu\otimes 1$ of $P$ define the principal series representation $\pi_{\varepsilon,\mu} $induced by it to be the left regular representation of $G$ on 
\[
\Ind_P^G(\varepsilon\otimes e^\mu\otimes 1)=\left\{f\in C^\infty(G)\,\mid f(gman)=|t|_\varepsilon^{-\mu-1}f(g), \,m\in M,\, a\in A,\, n\in N\right\},
\]
where $ma=(\begin{smallmatrix} t & 0 \\ 0 & t^{-1} \end{smallmatrix})\in MA$. We introduce the notation
\[
k_\theta=\begin{pmatrix}
\cos \theta & \sin \theta\\
-\sin\theta & \cos \theta 
\end{pmatrix},
\]
and $\zeta_m(k_\theta)=e^{im\theta}$.
According to the theory of Fourier series we have the $K$-type decomposition 
\[
\Ind_P^G(\varepsilon\otimes e^\mu \otimes 1)\cong \mathop{\widehat{{\bigoplus}}}_{m\in 2\ZZ+\varepsilon}\CC \zeta_m.
\]
We let $\Ind_P^G(\varepsilon\otimes e^\mu\otimes 1)_m$ denote the set of functions contained in the $K$-type given by $m\in \ZZ$, that is $\Ind_P^G(\varepsilon\otimes e^\mu\otimes 1)_m=\CC \zeta_m $. 

A basis of $\frakg$ is given by 
\begin{align*}
    H=\begin{pmatrix}
    1 & 0\\
    0 & -1
\end{pmatrix},\quad E=\begin{pmatrix}
    0 & 1\\
    0 & 0
\end{pmatrix},\quad F=\begin{pmatrix}
    0 & 0\\
    1 & 0
\end{pmatrix}.
\end{align*}
Consider the Casimir operator 
$$\Delta_\mu=d\pi(H)^2+d\pi(E+F)^2-d\pi(E-F)^2, $$
where $\pi=\pi_{\varepsilon,\mu}$.
\begin{proposition}[See e.g. {{\cite[Prop. 10.7]{C20}}}]\label{Eigenvalue Casimir}
For $f\in \Ind_P^G(\varepsilon\otimes e^{\mu}\otimes 1)$ we have
\begin{align*}
    \Delta_\mu f = (\mu^2-1)f.
\end{align*}
\end{proposition}

\begin{proposition}[See {{\cite[Prop. 10.8]{C20}}}] \label{Best of casselmand}
The representation $\Ind_P^G(\varepsilon\otimes e^\mu\otimes 1)$ is irreducible except when $\mu\in 1-\varepsilon-2\ZZ$ . If $\mu\in 1-\varepsilon-2\NN$ then $\Ind_P^G(\varepsilon\otimes e^\mu\otimes 1)$ decomposes as $V_0\oplus V_1\oplus V_2$ where $V_0$ is an irreducible representation containing exactly the $K$-types with $|m|\leq -\mu$. The quotient $\pi_{\varepsilon,\mu}^{\text{ds}}$ is a direct sum of two infinite dimensional representations $\pi_{\varepsilon,\mu}^\hol$ and $\pi_{\varepsilon,\mu}^\ahol$. 
\end{proposition}

\noindent Let $w_0=(\begin{smallmatrix} 0 & 1\\
-1 & 0
\end{smallmatrix}),$ a representative of the longest Weyl group element of $G$. Recall the definition of the Knapp--Stein intertwining operator 
\begin{align*}
    T_\mu^\varepsilon: \Ind_P^G(\varepsilon\otimes e^\mu \otimes 1)\rightarrow \Ind_P^G(\varepsilon\otimes e^{-\mu} \otimes 1),\quad 
    T_\mu^\varepsilon f(g) = \frac{1}{\Gamma(\frac{\mu+\varepsilon}{2})}\int_{\overline{N}} f(g w_0 \overline{n}) d\overline{n},
\end{align*}
for $\Re{(\mu)}>0$. The normalization is chosen such that $T_\mu^\varepsilon$ extends holomorphically to $\mu\in \CC $.

\begin{proposition} \label{norm resultat}
For $f\in \Ind_P^G(\varepsilon\otimes e^\mu\otimes 1)_m$ we have 
\[
T^\varepsilon_\mu f=b_m^\varepsilon (\mu) f,
\]
where 
\[
b_m^\varepsilon(\mu)=\sqrt{\pi}i^{[\varepsilon]_2} (-1)^{\frac{m+|m|}{2}-[\varepsilon]_2}\frac{\big(\frac{1+\varepsilon-\mu}{2}\big)_{\frac{|m|-\varepsilon}{2}}}{\Gamma\big(\frac{\mu+1+|m|}{2}\big)}.
\]
For $\varepsilon=0$ and $\mu\in 1-2\NN$ we have $b_m^0(\mu)\geq 0$ for all $m\in 2\ZZ$.
Whereas for $\varepsilon=1$, $m$ odd and $\mu\in -2\NN$ we have $-ib_m^1(\mu)\geq 0$ for $m>0$ and $ib_m^1(\mu)\geq 0$ for $m<0$.
\end{proposition}
\begin{proof}
As $T^\varepsilon_\mu$ maps $K$-types to $K$-types we have $T_\mu^\varepsilon f=T_\mu^\varepsilon f(e) f$. Now decompose 
\[
w_0\overline{n}_x=kan=\frac{1}{\sqrt{1+x^2}}
\begin{pmatrix}
x & 1\\
-1 & x
\end{pmatrix}
\begin{pmatrix}
\sqrt{x^2+1} & 0\\
0 & \frac{1}{\sqrt{1+x^2}}
\end{pmatrix}
\begin{pmatrix}
1 & \frac{x}{x^2+1}\\
0 & 1
\end{pmatrix},
\]
then applying $f$'s equivariance properties, we arrive at 
\[
\int_\RR f(w_0\overline{n}_x)\,dx=\int_\RR (x+i)^\frac{m-\mu-1}{2}(x-i)^\frac{-m-\mu-1}{2}\,dx=\frac{2^{1-\mu}\pi i^m\Gamma(\mu)}{\Gamma\big(\frac{\mu-|m|+1}{2}\big)\Gamma\big(\frac{\mu+|m|+1}{2}\big)},
\]
where in the last equality we used Lemma \ref{casselman integral}. Now dividing by $\Gamma(\frac{\mu+\varepsilon}{2})$ and shuffling around Gamma-factors we arrive at the result. 
\end{proof}

\noindent For $\mu\in i\RR $ we equip the space $\Ind_P^G(\varepsilon\otimes e^\mu\otimes 1)$ with the usual $L^2$-norm. 
Using Proposition \ref{norm resultat} we can for $\varepsilon=0$ and $\mu\in 1-2\NN$ equip $\Ind_P^G(0\otimes e^\mu\otimes 1)$ with the norm 
\[
\| f\|^2=\int_K f(k) \overline{T_\mu^0 f(k)}\,dk.
\]
Similarly for $\varepsilon=1$ and $\mu\in -2\NN$ we can equip $\Ind_P^G(1\otimes e^\mu\otimes 1)$ with the norm 
\[
\| f\|^2=\int_K f(k) \overline{\hat{T}_\mu^1 f(k)}\,dk
\]
where 
\[
\hat{T}_\mu^1 f=\begin{cases}
iT_\mu^1 f,& \text{for }m>0,\\
-iT_\mu^1 f,& \text{for }m<0\\
\end{cases}
\]
for $f\in \Ind_P^G(1\otimes e^\mu\otimes 1)_m$. The operator $\hat{T}_\mu^1$ is still an intertwining operator as it vanishes on $V_0$ per Proposition \ref{norm resultat} and thus we just altered it by a scalar on each of the summands in Proposition \ref{Best of casselmand}.

\section{The homogeneous space \texorpdfstring{$G/H$}{G/H}}
\noindent For a unitary character $\chi_{\varepsilon,\lambda}=\varepsilon\otimes e^\lambda$ with $\lambda\in i\RR$ the left-regular action $\tau_{\varepsilon,\lambda}$ of $G$ on the space of $L^2$-sections associated to the line bundle $G\times _H \CC_{\varepsilon,\lambda}\rightarrow G/H$, given by
\[
\Ind_H^G(\varepsilon\otimes e^\lambda)=
\left\{f: G\rightarrow \CC,\textit{ measurable}\mid f(gh)=\chi_{\varepsilon,\lambda}(h)^{-1}f(g),\, \int_{G/H}|f(g)|^2 d(gH)<\infty\right\},
\]
defines a unitary representation of $G$. The goal of this paper is to decompose this space. Furthermore we consider the subspace of compactly supported smooth functions
\[
C^\infty_c\text{-}\Ind_H^G(\varepsilon\otimes e^\lambda)=\left\{f\in C^\infty(G)\cap \Ind_H^G(\varepsilon\otimes e^\lambda)\mid \text{supp}(f)\subseteq\Omega,\,\Omega H\text{ is compact in }G/H \right\}.
\]
We will denote the smooth vectors in $\Ind_H^G(\varepsilon\otimes e^\lambda)$ by $\Ind_H^G(\varepsilon\otimes e^\lambda)^\infty$.
We introduce the notation 
$$b_u=\begin{pmatrix}
\cosh u & \sinh u\\
\sinh u & \cosh u 
\end{pmatrix},\quad
\overline{n}_x=\begin{pmatrix}
1 & 0 \\
x & 1
\end{pmatrix}.$$
Using the decomposition $G=KBA$, where $B=\{b_u\,|\,u\in \RR\},$ we consider $G/H$ in the global coordinates $(\theta,u)\in [0,\pi)\times \RR$ where $xH=k_\theta b_u H$ and the invariant measure is $d(xH)=\cosh(2u) dud\theta$, see e.g. \cite{M84}. 
Now in terms of these coordinates we have the $K$-type decomposition
\begin{align}\label{K-types}
    C^\infty_c\text{-}\Ind_H^G(\varepsilon\otimes e^\lambda)= \mathop{\widehat{\bigoplus}}_{m\in 2\ZZ+\varepsilon}\CC \zeta_m \otimes C_c^\infty(\RR)
\end{align} 
with $\zeta_m(k_\theta)=e^{im\theta}$.
We let $\Ind_H^G(\varepsilon\otimes e^\lambda)_m$ denote the set of functions contained in the $K$-type given by $m\in 2\ZZ+\varepsilon$, that is $\Ind_H^G(\varepsilon\otimes e^\lambda)_m=\CC \zeta_m \otimes C_c^\infty(\RR)$.

We denote by $\Delta_\lambda$ the Casimir operator for the representation $\Ind_H^G(\varepsilon\otimes e^\lambda)$ defined in a similar fashion as for the principal series. 
\begin{proposition} \label{Casimir}
Written in the coordinates $(\theta, u)$ the Casimir operator $\Delta_\lambda$ is given by

\begin{align*}
    \Delta_\lambda=\frac{\lambda^2}{\cosh^2(2u)}+2\lambda\frac{\tanh(2u)}{\cosh(2u)}\partial_\theta + 2\tanh(2u)\partial_u-\frac{1}{\cosh^2(2u)}\partial_\theta^2+\partial_u^2.
\end{align*}
\end{proposition}
\begin{proof}
This is a standard computation.
\end{proof}

Another set of coordinates can be obtained by using the Iwasawa decomposition $G=KA\overline{N}$ with $(\theta,y)\in [0,\pi)\times\RR$ where $xH=k_\theta \overline{n}_yH$. The invariant measure is given by $d(xH)=\frac{1}{2}dyd\theta$ see \cite[Chap. 5, §6]{knapp2016representation}.

\section{Constructing an isomorphism}

\noindent The goal of this section is to construct the explicit isomorphism in the following theorem, of which the proof was in large presented to us by Jan Frahm.
\begin{theorem}\label{indre Ind}
For  $\nu,\lambda\in i\RR$ the map 
\[
\ankh_\lambda^\nu:\Ind_H^G(\varepsilon\otimes e^\lambda)\rightarrow \Ind_H^G(\varepsilon\otimes e^\nu)
\]
given by
\[
\ankh_\lambda^\nu f(g)=\frac{1}{\sqrt{\pi}2^{\frac{\lambda-\nu}{4}}{}}\frac{\Gamma\left(\frac{2+\nu-\lambda}{4}\right)}{\Gamma \left (\frac{\lambda-\nu}{4}\right)}\int_\RR |x|^{ \frac{\lambda-\nu}{2}-1} f(g\overline{n}_x)\,dx
\]
defines a unitary isomorphism intertwining $\Ind_H^G(\varepsilon\otimes e^\lambda)$ and $\Ind_H^G(\varepsilon\otimes e^\nu)$.
\end{theorem}
\noindent To this extent we consider the minimal parabolic subgroup $\overline{P}=\overline{N}AM\subset G$ and let
$$\Ind_{MA}^P(\varepsilon\otimes e^\lambda)=\left\{f:\overline{N}AM\to \CC \, \vert \, f(gma)=\sgn(m)^\varepsilon a^{-\lambda-1} f(g),\, \& \,\int_{P/MA}|f(g)|^2\,dg<\infty \right \}, $$
where $a^{\lambda}:=e^{\lambda(X)}$ for $a=e^X$ and $\lambda\in \fraka_\CC^*\cong \CC$.
\begin{lemma}[Induction in stages]\label{folklore}
$$\Ind _{MA}^G(\varepsilon\otimes e^\lambda)\simeq \Ind_{MA\overline{N}}^G(\Ind_{MA}^{MA\overline{N}}(\varepsilon\otimes e^\lambda)), $$
where the map is given by $f\mapsto F$ where $F(g)(\overline{p})=f(g\overline{p})$ and thus the inverse is given by $f(g)=F(g)(1)$. 
\end{lemma}
\begin{proof}
See e.g. \cite[Chapter VI, section 9]{gaal2012linear} 
\end{proof}

\begin{proof}[Proof of Theorem \ref{indre Ind}]
We first show the isomorphism claim of Theorem \ref{indre Ind}. By Lemma \ref{folklore} it suffices to show that $\Ind_{MA}^{\overline{P}}(\varepsilon\otimes e^\lambda)\simeq \Ind_{MA}^{\overline{P}}(\varepsilon\otimes e^\nu)$.
Let $\text{rest}_{\overline{N}}:\Ind_{MA}^{\overline{P}}(\varepsilon\otimes e^\lambda)\to L^2(\overline{N})$ be the restriction from $\overline{P}$ to $\overline{N}$. We let $\Phi$ be the inverse map which is given by $\Phi F(\overline{n}am)=\sgn(m)^\varepsilon a^{-\lambda-1} F(\overline{n})$. Let $\pi_{\varepsilon,\lambda}$ be the left regular representation on $\Ind_{MA}^{\overline{P}}(\varepsilon\otimes e^\lambda)$ and define $\tilde{\pi}_{\varepsilon,\lambda}(g)=\text{rest}_{\overline{N}}\circ \pi_{\varepsilon,\lambda}(g) \circ \Phi$. Then $\overline{P}$ acts on $L^2(\overline{N})$ via $\tilde{\pi}_{\varepsilon,\lambda}$, and the above statement reduces to showing that $\tilde{\pi}_{\varepsilon,\lambda}\cong \tilde{\pi}_{\varepsilon,\nu}$ for $\lambda,\nu \in i\RR$.

To construct an isomorphism $H:L^2(\RR)\rightarrow L^2(\RR)$ intertwining $\tilde{\pi}_{\varepsilon,\lambda}$ and $\tilde{\pi}_{\varepsilon,\nu}$ we note that the action of $P=\overline{N}AM$ on $f\in L^2(\overline{N})$ is given by  
\begin{align*}
    &\tilde{\pi}_{\varepsilon,\lambda}(\overline{n})f(\overline{n}')=f(\overline{n}^{-1} \overline{n}')\qquad\qquad\qquad\qquad\qquad\quad \overline{n},\overline{n}'\in \overline{N},\\
    &\tilde{\pi}_{\varepsilon,\lambda}(ma)f(\overline{n})=\sgn(m)^\varepsilon a^{\lambda+1} f\big((ma)^{-1}\overline{n}(ma)\big),\quad m\in M,\, a\in A,\, \overline{n}\in \overline{N}.
\end{align*}
Identifying $\overline{N}\simeq \RR $, $M\simeq \{\pm 1\}$ and $A\simeq \RR_{>0}$, the above becomes
\begin{align*}
    &\tilde{\pi}_{\varepsilon,\lambda}(y)f(x)=f(x-y),\qquad x,y\in \RR,\\
    &\tilde{\pi}_{\varepsilon,\lambda}(t)f(x)=t^{\lambda+1}f(t^2x),\quad x\in \RR,\ t\in \RR_{>0}.
\end{align*}
 Since $\overline{N}$ acts by translation any such intertwining operator $H$ must be a translation invariant operator on $L^2(\RR)$, hence there exists some tempered distribution $u \in \mathcal{S}'(\RR)$ such that $H$ is given by convolution with $u$, that is $HF(x)=\langle u, \tau_x \check{F}\rangle$, where $\tau_x F(y)=F(y-x)$ and $\check{F}(x)=F(-x)$. 
Furthermore let $d_tf$ denote the dilation of $f$ by $t\in \RR\setminus\{0\}$, i.e.  $d_tf(x)=f(tx)$, then 
\begin{align*}
    H\circ \tilde{\pi}_{\varepsilon,\lambda}(t)F(x) = \tilde{\pi}_{\varepsilon,\nu}(t)\circ H F(x).
\end{align*}
Evaluating at $x=0$ then yields $\langle d_{t^{-2}} u,\check{F}\rangle = t^{\nu-\lambda+2}\langle u ,\check{F}\rangle$. Hence $u$ is a homogeneous distribution of degree $\frac{\lambda-\nu-2}{2}$ and we conclude that
\[
Hf=f*|x|^{\frac{\lambda-\nu-2}{2}}_\delta \quad \text{ for some } \delta\in \{0,1\}.
\]
Comparing with Lemma \ref{convolution constant} we see that these are both necessary and sufficient conditions for $H$ to establish an isomorphism between $\Ind_{MA}^P(\varepsilon\otimes e^\lambda)$ and $ \Ind_{MA}^P(\varepsilon\otimes e^\nu)$. Putting $\delta = 0$, composing with the map from Lemma \ref{folklore} then yields the desired isomorphism. To see that the normalization indeed makes $\ankh_\lambda^\nu$ unitary, it suffices to note that for $\lambda,\nu \in i\RR$ Lemma \ref{convolution constant} gives

\[\left|\frac{1}{\sqrt{\pi}2^{\frac{\lambda-\nu}{4}}{}}\frac{\Gamma\left(\frac{2+\nu-\lambda}{4}\right)}{\Gamma \left (\frac{\lambda-\nu}{4}\right)}\mathcal{F}(|x|^{\frac{\lambda-\nu}{2}-1})\right|=1.
\]
\end{proof}

\section{Eigenfunctions for the Casimir operator}

\noindent By Theorem \ref{indre Ind} a Plancherel formula on $\Ind_H^G(\varepsilon \otimes e^\lambda)$ for some fixed $\lambda\in i\RR$ can be extended to all $\nu\in  i\RR$ by compositon with the unitary isomorphism $\ankh^\nu_\lambda$ from Theorem \ref{indre Ind}. Following this we will therefore mostly consider the cases of which $\lambda =0$, which often simplifies matters considerably.

For $f\in \Ind_H^G(\varepsilon\otimes e^0)_m$ with $f(k_\theta b_u)=e^{im\theta}\cdot h(u),\ h\in C_c^\infty(\RR)$ we have 
\begin{align*}
    \Delta_0 f = e^{im\theta}\tilde{\Delta}_{m} h(u),
\end{align*}
for some differential operator $\tilde{\Delta}_{m}$. 
\begin{lemma}\label{Casimir on m part}
Let $h\in C_c^\infty(\RR)$ and $m\in \ZZ$. Then we have
\begin{equation*}
    \tilde{\Delta}_m\cosh^{\frac{m}{2}}(2u)h\big(-\sinh^2(2u)\big) = \cosh^{\frac{ m}{2}}(2u)\square_{ m} h\big(-\sinh^2(2u)\big)\label{Casimir diff operator}.
\end{equation*}
For $t=-\sinh^2(2u)$ the operator $\square_{m}$ is given by
\begin{align*}
    \square_{m} = m(m+2)+8(-1+\left(3+ m\right)t)\frac{d}{dt}-16t\cdot(1-t)\frac{d^2}{dt^2}.\label{hypergeo Casimir}
\end{align*}
\end{lemma}
\begin{proof}
Follows directly from Proposition \ref{Casimir}.
\end{proof}
\noindent Recall that a hypergeometric differential equation has the form
\[
t(1-t)\frac{d^2}{dt^2}+[c-(a+b+1)t]\frac{d}{dt}-ab=0.
\]
If $c$ is not a non-positive integer there are two independent solutions (around $t=0$) 
\[
\ _2F_1(a,b;c;t)\quad\text{and}\quad t^{1-c}\ _2F_1(1+a-c,1+b-c;2-c;t),
\]
expressed in terms of the hypergeometric function $\,_2 F _1(a,b;c;t)$.

We note that the eigenvalue problem $\square_m f=(\mu^2-1)f$ is a hypergeometric differential equation thus giving us two linearly independent solutions 
$\varphi^{m}_\mu$ and $\psi_\mu^{m}$. Using the notation from appendix \ref{Fourier Jacobi} we can express these solutions as
\[
\varphi^{m}_\mu(u)=\phi_{\frac{-i\mu}{2}}^{-\frac{1}{2},\frac{ m}{2}}(u),\quad \psi_\mu^{m}(u)=i\sinh(u)\cdot   \phi_{\frac{-i\mu}{2}}^{\frac{1}{2},\frac{m}{2}}(u), 
\]
where $u\in [0,\infty)$. Note that these functions allow for natural extensions from $[0,\infty)$ to  $\RR$. We now restate the results from Appendix \ref{Fourier Jacobi} in terms of $\varphi_\mu^m$ and $\psi_\mu^m$.
\begin{proposition}
\label{FF}
For $f\in C^\infty_c([0,\infty))$,
let $(J_j f)(\mu)$, $j=0,1,$ denote the Fourier--Jacobi transforms of $f$ given by
\begin{align*}
    J_0 f(\mu) &=\int_0^\infty f(t)\varphi_{\mu }^{m}(t) \cosh^{m+1}(t)\, dt,\\
    J_1 f(\mu) &= \int_0^\infty f(t)\psi_{\mu }^{m}(t) \sinh(t)\cosh^{ m+1}(t)\, dt.
\end{align*}
Then we have the following inversion formulas
\begin{align*}
    f(t) &= \frac{1}{4\pi^2}\int_{i\RR} J_0 f(\mu)\varphi^{m}_{\mu}(t) \frac{d\mu}{|\ell_0(\mu)|^{2}} -\frac{1}{2\pi}\sum\limits_{\mu \in D_0} J_0 f(\mu)\varphi_{\mu}^{m}(t)\underset{\nu=\mu}{\Res}(\ell_0(\nu)\ell_0(-\nu))^{-1}\\
    \sinh(t)f(t) &= \frac{-1}{\pi^2}\int_{i\RR} J_1 f(\mu)\psi^{ m}_{\mu}(t) \frac{d\mu}{|\ell_1(\mu)|^{2}} +\frac{2}{\pi}\sum\limits_{\mu \in D_1} J_1f(\mu)\psi_{\mu}^{m}(t)\underset{\nu=\mu}{\Res}(\ell_1(\nu)\ell_1(-\nu))^{-1}.
\end{align*}
with $D_j=\{\eta \in \RR\mid \eta=4k+1+2j-|m|<0, \, k\in \NN_0 \}$ and

\begin{align*}
    \ell_j(\mu) = \frac{\Gamma(\frac{\mu}{2})}{\Gamma(\frac{\mu+1+2j+|m|}{4})\Gamma(\frac{\mu+1+2j-|m|}{4})}.
\end{align*}
\end{proposition}
\begin{Remark}\label{delta er ligemeget-}
As $\varphi_\mu^m$ and $\psi_\mu^m$ are given in terms of hypergeometric functions we get 
\[
\varphi_\mu^m=\varphi_{-\mu}^m,\quad\text{and}\quad \psi_\mu^m=\psi_{-\mu}^m,
\]
as$\ _2F_1(a,b;c;t)=\ _2F_1(b,a;c;t)$.
The Euler transformation$\ _2F_1(a,b;c;t)=(1-t)^{c-a-b} \ _2F_1(c-a,c-b;c;t)$ amounts to 
\begin{align*}
    \varphi_\mu^{m}(u)=\cosh^{- m}(u)\varphi_\mu^{- m}(u),\quad\text{and}\quad \psi_\mu^{m}(u)=\cosh^{- m}(u)\psi_\mu^{- m}(u).
\end{align*}
\end{Remark}

\section{Intertwining operators}
\noindent To obtain an explicit Plancherel formula for representation theoretic purposes, we require expressions for intertwining operators between the representation spaces introduced in earlier sections. More explicitly we consider intertwining operators 
\begin{align*}
    P:\Ind_P^G(\delta \otimes e^\mu \otimes 1)\rightarrow \Ind_H^G(\varepsilon\otimes e^\lambda)^\infty,\\
    A:C_c^\infty\mbox{--}\Ind_H^G(\varepsilon\otimes e^\lambda)\rightarrow \Ind_P^G(\delta \otimes e^\mu \otimes 1)
\end{align*}
and their realizations in terms of the coordinates introduced in earlier sections. Such operators only exist when $\varepsilon=\delta$ as $M$ lies in the center of $G$.

We fix $\varepsilon\in \ZZ /2\ZZ$ and supress it in the notation for the rest of this section. When $\varepsilon$ appear in formulas we will consider it as number in $\{0,1\}$ where we will use the notation $[\,\cdot\,]_2$ when confusions can occur.

For $\xi\in \ZZ /2\ZZ $ and $\lambda\in i\RR$ consider the kernel
\[
K_{\lambda,\mu}^\xi(g)=|g_{11}|^{\frac{\lambda+\mu-1}{2}}_{\xi+\varepsilon}|g_{21}|^{\frac{\mu-\lambda-1}{2}}_\xi,\quad g\in G,
\]
where $g_{ij}$ is the $(i,j)$'th entry in $G$. As $g_{11}$ and $g_{21}$ does not simultaneously vanish in $G$ this kernel enjoys many of the similar properties as a Riesz distribution (see Appendix \ref{Fourier transform and Riesz distributions}). $\Gamma(\frac{\mu+1}{2})^{-1}K_{\lambda,\mu}^\xi$ is locally integrable for $\Re{(\mu)}>-1$ and admits a holomorphic continuation as a distribution to $\mu \in \CC.$  
\begin{proposition} \label{definition af intertwiner}
The map given by
\[
P_{\lambda,\mu}^\xi f(g)=\int_{K} K_{\lambda, \mu}^\xi(g^{-1}k)f(k)\,dk,
\]
defines an intertwining operator $\Ind_P^G(\varepsilon \otimes e^\mu \otimes 1)\rightarrow \Ind_H^G(\varepsilon\otimes e^\lambda)^\infty$.
Similarly the map given by 
\[
A^\xi_{\lambda,\mu}f(g)=\int_{G/H} K_{-\lambda,-\mu}^\xi(x^{-1} g) f(x) \,d(xH),
\]
defines an intertwining operator $C_c^\infty \mbox{--}\Ind_H^G(\varepsilon\otimes e^\lambda)\rightarrow \Ind_P^G(\varepsilon \otimes e^\mu \otimes 1)$. Both integrals should be understood in the distributional sense. 
\end{proposition}
\begin{proof}
The equivariance properties follows by direct verification. 
\end{proof}
\begin{proposition} \label{Tony transform}
For $\xi,\varepsilon \in \ZZ/2\ZZ$ we have the following relation
\begin{align*}
    T_\mu^\varepsilon \circ A_{\lambda,\mu}^\xi =d^\xi_{\lambda,\mu}  A_{\lambda,-\mu}^{\xi+\varepsilon},
\end{align*}
where 
\[
d_{\lambda,\mu}^\xi=(-1)^{\lfloor \frac{\varepsilon+\xi}{2}\rfloor}\sqrt{\pi}\frac{\Gamma\big(\frac{1-\lambda-\mu+2[\xi+\varepsilon]_2}{4}\big)\Gamma\big(\frac{1+\lambda-\mu+2\xi}{4}\big)}{\Gamma\big(\frac{1+\lambda+\mu+2[\xi+\varepsilon]_2}{4}\big)\Gamma\big(\frac{1-\lambda+\mu+2\xi}{4}\big)\Gamma\big(\frac{1-\mu+\varepsilon}{2}\big)}.
\]
\end{proposition}

\begin{proof}
Fix $g\in G$ and put $z=x^{-1}g$. Then the set $\{xH\in G/H\mid z_{11}z_{21}= 0\}$ is a $d(xH)$-null set. Using Lemma \ref{convolution constant} for $0<\Re{(\mu)}<1$ we get 

\begin{align*}
    \int_{\overline{N}} K_{-\lambda,-\mu}^\xi(zw_0\overline{n}) d\overline{n}
    &= |z_{11}|_{\xi+\varepsilon}^{\frac{-\mu-\lambda-1}{2}}|z_{21}|_{\xi}^{\frac{\lambda-\mu-1}{2}}\int_\RR|x|_{\xi+\varepsilon}^{{\frac{-\lambda-\mu-1}{2}}}|x-\frac{1}{z_{11}z_{21}}|_{\xi}^{\frac{\lambda-\mu-1}{2}}dx\\
    &=\Gamma\big(\frac{\mu+\varepsilon}{2}\big)d^\xi_{\lambda,\mu} K_{-\lambda,\mu}^{\xi+\varepsilon}(z),\quad a.e.
\end{align*}
The claim then follows by analytic continuation.
\end{proof}
We introduce the notation 
\[
\omega_m^\xi=(-1)^\xi+(-1)^m i^m. 
\]
Note that $\omega_{-m}^\xi=\overline{\omega}_m^\xi$ and as $\varepsilon\equiv m \mod 2$ we have 
\[
\omega_m^\xi=
\begin{cases}
(-1)^\xi+(-1)^\frac{m}{2},& \varepsilon=0,\\
(-1)^{\xi}+i(-1)^\frac{m+1}{2},&\varepsilon=1,
\end{cases}
\quad\text{and}\quad
\omega_m^0\overline{\omega}_m^1=
\begin{cases}
0,&\varepsilon=0,\\
2i(-1)^\frac{m-1}{2}, &\varepsilon=1.
\end{cases}
\]
\begin{proposition}
\label{P image}
For $\mu\in \CC$ and $m\in \ZZ$ we have 
$$\frac{1}{\Gamma\big(\frac{\mu+1}{2}\big)}P_{\mu}^\xi\zeta_m(k_\theta b_u)=\zeta_m(k_\theta) \cosh^{\frac{m}{2}}(2u) \left(\omega_m^\xi c_m(\mu) \varphi_\mu^{ m}(2u)+\frac{i}{2}\omega_m^{\xi+1} c_m(\mu-2)\psi^{ m}_{\mu}(2u)\right),$$
where
\begin{align*}
    c_m(\mu)= \frac{2^{1-\mu}\pi e^{i\frac{m\pi}{4}}}{\Gamma(\frac{\mu+3+|m|}{4})\Gamma(\frac{\mu+3-|m|}{4})}.
\end{align*}
\end{proposition}
\begin{proof}
As $P_\mu^\xi$ intertwines $\pi_{\varepsilon,\mu}$ and $\tau_{\varepsilon,0}$ 
it also intertwines the derived representations $d\pi_{\varepsilon,\mu}$ and $d\pi_{\varepsilon,0}$. Hence $P_\mu^\xi$ intertwines $\Delta_0$ and $\Delta_\mu$ and therefore the image of $P_\mu^\xi$ is contained in the eigenspace of $\Delta_0$ to the eigenvalue $\mu^2-1$ by Proposition \ref{Eigenvalue Casimir}. Fix $\mu$ with $\Re{(\mu)}>1$. From Lemma \ref{Casimir diff operator} it follows for generic $\mu$ that
\begin{align*}
    P_\mu^\xi\zeta_m(k_\theta b_u) =\cosh^{\frac{m}{2}}(2u)\zeta_m(k_\theta)\big( a_m^\xi(\mu)\cdot \varphi_\mu^{m}(2u) + b^\xi_m(\mu)\cdot \psi_\mu^{m}(2u)\big)
\end{align*}
for some $a_m^\xi(\mu),b_m^\xi(\mu)\in \CC$.
Hence, it only remains to compute $a_m^\xi(\mu)$ and $b^\xi_m(\mu)$. Note that $\varphi_\mu^m(0)=1$ and $\psi_\mu^m(0)=0$ and hence $P_\mu^{\xi}\zeta_m(k_\theta)=\zeta_m(k_\theta)a_m^\xi(\mu)$ so
\begin{align*}
    a_m^\xi(\mu) &=P_\mu^{\xi}\zeta_m(e) = \int_K K^\xi_{\mu}(k_{\theta})\zeta_m(k_\theta) dk_\theta=2\int_0^\pi|\cos \theta |_{\xi+\varepsilon}^\frac{\mu-1}{2}|-\sin \theta |^\frac{\mu-1}{2}_\xi e^{i m\theta }\,d\theta \\
    &=2^{\frac{1-\mu}{2}}\omega_m^\xi \int_{0}^{\pi} (\sin \theta)^{\frac{\mu-1}{2}}e^{i\frac{m}{2}\theta}d\theta
    =\frac{2^{1-\mu}\pi\omega_m^\xi e^{i\frac{m\pi}{4}}\Gamma(\frac{\mu+1}{2})}{\Gamma(\frac{\mu+3+|m|}{4})\Gamma(\frac{\mu+3-|m|}{4})},
\end{align*}
by Lemma \ref{appendix trig integral}.

To compute $b_m^\xi(\mu)$ it suffices to note that $\frac{d}{du}\varphi_{\mu}^{m}(2u)\mid_{u=0}=0$ and $\frac{d}{du}\psi_{\mu}^{ m}(2u)\mid_{u=0}\-=2i$, hence
\begin{align*}
    2i\cdot b_m^\xi(\mu)&=\frac{d}{du}P_\mu^\xi \zeta_m( b_u)\big|_{u=0}=\int_K \frac{d}{du} K_\mu^\xi(b_{-u}k_{\theta})\big|_{u=0}\zeta_m(k_\theta) dk_\theta\\
    &=\frac{1-\mu}{2}\int_KK_{\mu-2}^{\xi+1}(k_{\theta})\zeta_{m}(k_\theta)dk_\theta\\
    &= \frac{1-\mu}{2}a_m^{\xi+1}(\mu-2),
\end{align*}
from which the result follows by analytic continuation. 
\end{proof}
\noindent Let $f\in \Ind_H^G(\varepsilon \otimes e^0)_m$ and write $f(k_\theta b_u)=\zeta_m(k_\theta)h(u)$ for some $h\in L^2(\RR,\cosh(2u) du)$. Now let $h_e(u)$ be the even part of $h$ and $h_o(u)$ the odd part. We introduce the following notation 
\begin{align*}
    J_0 f(\mu,\theta) &=\zeta_m(k_\theta)J_0(\cosh^{-\frac{ m}{2}}(x)h_{e}( b_\frac{x}{2}))(\mu),\\
    J_1 f(\mu,\theta) &= \zeta_m(k_\theta)J_1(\sinh^{-1}(x)\cosh^{-\frac{ m}{2}}(x)h_{o}(b_{\frac{x}{2}}))(\mu)
\end{align*}
where the $x$ denotes the variable the Fourier--Jacobi transform is done with respect to. 
\begin{proposition}\label{A image}
Let $f\in \Ind_H^G(\varepsilon \otimes e^0)_m$ then
\[
\frac{1}{\Gamma\big(\frac{1-\mu}{2}\big)} A_\mu^\xi f(k_\theta) =\frac{\overline{\omega}_m^\xi}{2} c_{-m}(-\mu)J_0 f(\mu,\theta) +i\frac{\overline{\omega}_m^{\xi+1}}{4}c_{-m}(-\mu-2)J_1 f(\mu,\theta).
\]
\end{proposition}

\begin{proof}
As $A^\xi_{\mu}$ is an intertwining operator, it maps $K$-types to $K$-types, thus $A_\mu^\xi f(k_\theta)=A_\mu^\xi f(e) \times \zeta_m(k_\theta)$. Now 
\[
A_\mu^\xi f(e) = \int_0^\pi \int_\RR K_{-\mu}^\xi(b_u^{-1}k_{\theta}^{-1})f(k_\theta b_u)\cosh(2u)du d\theta
=\frac{1}{2}\int_\RR \cosh(2u)h(u) P_{-\mu}^\xi\zeta_{-m}(b_u)du.
\]
From Proposition \ref{P image} and Remark \ref{delta er ligemeget-} we have
\begin{align*}
    \frac{A_\mu^\xi f(e)}{\Gamma(\frac{1-\mu}{2})}&=\frac{1}{2}\int_\RR \! h(u)\cosh^{-\frac{m}{2}+1}(2u) \Big(\omega_{-m}^{\xi}c_{-m}(-\mu) \varphi_{-\mu}^{ -m}(2u)+\frac{i}{2}\omega^{\xi+1}_{-m}c_{-m}(-\mu-2)\psi^{-m}_{-\mu}(2u)\Big)du\\
    &=\int_0^\infty h_e(u)\cosh^{\frac{m}{2}+1}(2u) \overline{\omega}^{\xi}_{m} c_{-m}(-\mu) \varphi_{\mu}^{ m}(2u)du\\
    &\phantom{=}+\frac{i}{2}\int_0^\infty h_{o}(u)\cosh^{\frac{ m}{2}+1}(2u)\overline{\omega}^{\xi+1}_{m}c_{-m}(-\mu-2)\psi^{ m}_{\mu}(2u)du\\
    &=\frac{1}{2}\overline{\omega}^{\xi}_{m}c_{-m}(-\mu) (J_0 \cosh^{-\frac{ m}{2}}(u)h_{e}(\tfrac{x}{2}))(\mu)\\
    &\phantom{=}+\frac{i}{4}\overline{\omega}_m^{\xi+1}c_{-m}(-\mu-2)(J_1\sinh^{-1}(x)\cosh^{-\frac{ m}{2}}(x)h_{o}(\tfrac{x}{2}))(\mu).\qedhere
\end{align*}
\end{proof}
\noindent Combining Proposition \ref{P image} and Proposition \ref{A image} yields an explicit intertwining operator
\begin{align*}
    \frac{P_\mu^\xi A^{\xi'}_\mu}{\Gamma(\frac{1+\mu}{2})\Gamma(\frac{1-\mu}{2})}:C_c^\infty\mbox{--}\Ind_H^G(\varepsilon\otimes e^0)_m\rightarrow \Ind_H^G(\varepsilon\otimes e^0)_m^\infty.
\end{align*}
By Propositions \ref{P image} and \ref{A image}, the above intertwining operator is holomorphic in $\mu$, i.e. the above defines a holomorphic family of intertwining operators, intertwining $\Ind_H^G(\varepsilon\otimes e^0)_m$ with itself.

\section{Combining intertwining operators} \label{composing}
\noindent In this section we consider a function $f\in C_c^\infty\mbox{--}\Ind_H^G(\varepsilon\otimes e^0)_m$ and then write
\[
f(k_\theta b_{\frac{u}{2}})=\cosh^{\frac{m}{2}}(u)\bigg[\big(\cosh^{-\frac{m}{2}}(u)f_e(k_\theta b_{\frac{u}{2}})\big)+\sinh(u)\big(\cosh^{-\frac{m}{2}}(u)\sinh^{-1}(u)f_o(k_\theta b_{\frac{u}{2}})\big)\bigg],
\]
then apply the two inversion formulas from Proposition \ref{FF} to each of the two terms giving
\begin{multline*}
    f(k_\theta b_{\frac{u}{2}})=\cosh^{\frac{m}{2}}(u)\bigg[
    \frac{1}{\pi^2}\int_{i\RR}J_0 f(\mu,\theta)\varphi_\mu^m(u)\frac{d\mu}{4|\ell_0(\mu)|^2}-\frac{1}{\pi^2}\int_{i\RR} J_1 f(\mu,\theta)\psi_\mu^m(u)\frac{d\mu}{|\ell_1(\mu)|^2}\\
    -\frac{1}{2\pi}\!\sum_{\mu\in D_0}\!J_0 f(\mu,\theta)\varphi_\mu^m(u)\underset{\nu=\mu}{\Res}\big(\ell_0(\nu)\ell_0(-\nu)\big)^{-1}+\frac{2}{\pi}\!\sum_{\mu\in D_1}\! J_1 f(\mu,\theta)\psi_\mu^m(u)\underset{\nu=\mu}{\Res}\big(\ell_1(\nu)\ell_1(-\nu)\big)^{-1}
    \bigg].
 \end{multline*}
The goal is then to express this decomposition in terms of some combination of the operators $P_\mu^\xi A^{\xi'}_\mu f(k_\theta b_{\frac{u}{2}})$ which by a quick glance at Propositions \ref{P image} and \ref{A image} appears plausible. The following identity will be used multiple times in the following subsections

\begin{align}\label{c og l}
2^4c_m(\nu)c_{-m}(-\nu)\ell_0(\nu)\ell_0(-\nu)&=c_m(\nu-2)c_{-m}(-\nu-2)\ell_1(\nu)\ell_1(-\nu) \nonumber\\
&=\frac{2^5(-1)^{1+\varepsilon}}{\pi} \frac{\cos^2\big(\frac{\pi(\nu+\varepsilon)}{2}\big)}{\nu\sin\big(\frac{\pi\nu}{2}\big)},
\end{align}
which follows from Gamma-function identities and recalling that $m\equiv \varepsilon\mod 2$.

\subsection{The continuous part}
\noindent For $\mu \in \CC$ we introduce the following maps
\begin{align*}
    \textbf{P}^\xi_{\mu} =
    \frac{P_\mu^\xi}{\Gamma(\frac{\mu+1}{2})},\quad  \textbf{A}_\mu^\xi= \frac{A_\mu^\xi}{\Gamma(\frac{1-\mu}{2})},
\end{align*}
which are holomorphic in $\mu$.
\begin{proposition}
We have 
\begin{multline*}
    \sum_{\xi=0}^1\mathbf{P}_\mu^\xi \mathbf{A}_\mu^\xi f(k_\theta b_u)\\=\cosh^{\frac{m}{2}}(2u)\bigg[2c_m(\mu)c_{-m}(-\mu)J_0 f(\mu,\theta)\varphi_\mu^m(2u)
    -\frac{1}{2}c_m(\mu-2)c_{-m}(-\mu-2)J_1 f(\mu,\theta)\psi_\mu^m(2u)\bigg].
\end{multline*}
Combining this with (\ref{c og l}) we get

\begin{multline*}
    \int_{i\RR}\sum_{\xi=0}^1\mathbf{P}_\mu^\xi \mathbf{A}_\mu^\xi f(k_\theta b_{\frac{u}{2}})\frac{d\mu}{|a(\mu)|^2}\\
    =\cosh^{\frac{m}{2}}(u)\bigg[
    \frac{1}{\pi^2}\int_{i\RR}J_0 f(\mu,\theta)\varphi_\mu^m(u)\frac{d\mu}{4|\ell_0(\mu)|^2}-\frac{1}{\pi^2}\int_{i\RR} J_1 f(\mu,\theta)\psi_\mu^m(u)\frac{d\mu}{|\ell_1(\mu)|^2}\bigg ]
\end{multline*}
where 
\begin{align*}
    a(\mu)= 2^{\frac{3}{2}}\pi\frac{\Gamma(\frac{\mu}{2})}{\Gamma(\frac{1+\mu+\varepsilon}{2})\Gamma(\frac{1+\mu-\varepsilon}{2})}.
\end{align*}

\end{proposition}
\begin{proof}
When computing $\sum_{\xi=0}^1\mathbf{P}_\mu^\xi \mathbf{A}_\mu^\xi f(k_\theta b_\frac{u}{2})$ we apply Proposition \ref{P image} and \ref{A image}.  We obtain some cross-terms, containing factors like $J_0 f(\mu,\theta)\psi_\mu^m(u)$, but since 
\begin{align*}
        \sum\limits_{\xi=0}^1\omega_m^\xi\overline{\omega}_{m}^{\xi+1}=0 \text{ and } \sum\limits_{\xi=0}^1\omega_m^\xi\overline{\omega}_{m}^{\xi}= 4,
    \end{align*}
    no cross-terms survive and the assertion follows. 
\end{proof}

\noindent To express the discrete part in terms of $P_\mu^\xi A_\mu^{\xi'} f$ is a bit more delicate as we cannot simply take a sum to make the cross terms disappear thus we need to make a suitable choice of normalization. The cases for $\varepsilon=0$ and $\varepsilon=1$ will be treated differently and the main culprit as to why is the factor $\omega_m^\xi$ which for $\varepsilon=0$ vanishes depending on the parity of $\frac{m}{2}$ and for $\varepsilon=1$ never vanishes.

\subsection{The discrete part for \texorpdfstring{$\varepsilon=0$}{epsilon=0}} 
\noindent In this subsection we fix $\varepsilon=0$.  
Consider the following normalizations 
\begin{align*}
\widehat{P}^{\xi}_{\mu}=\frac{\Gamma(\frac{\mu+3-2\xi}{4})}{\Gamma(\frac{\mu+1}{2})\Gamma(\frac{\mu+1+2\xi}{4})}P^{\xi}_{\mu},\quad  \widehat{A}^{\xi}_{\mu}=\frac{\Gamma(\frac{-\mu+3-2\xi}{4})}{\Gamma(\frac{-\mu+1}{2})\Gamma(\frac{-\mu+1+2\xi}{4})}A^{\xi}_{\mu},
\end{align*}
which, by the duplication formula for the Gamma-function, does not introduce any poles. Now introduce the operators

\begin{align*}
    \mathbb{P}_{\mu}
:=\widehat{P}^0_{\mu}+ \widehat{P}^1_{\mu} \quad\text{and}\quad \mathbb{A}_{\mu}
:=\widehat{A}^0_{\mu}+ \widehat{A}^1_{\mu}.
\end{align*}
\begin{lemma}\label{Lemma 1 - disk}
For a fixed $m\in 2\ZZ$ we have 
\[
\mathbb{P}_{\mu}\zeta_m(k_\theta b_u)
=\zeta_m(k_\theta)\cosh^{\frac{ m}{2}}(2u)\big(\alpha_m(\mu)\varphi_{\mu}^{m}(2u)+\beta_m(\mu)\psi_{\mu}^{m}(2u)\big), 
\]
where 
\begin{align*}
    \alpha_m(\mu)=c_m(\mu)\left(\omega_m^0\frac{\Gamma(\frac{\mu+3}{4})}{\Gamma(\frac{\mu+1}{4})}+\omega_m^1\frac{\Gamma(\frac{\mu+1}{4})}{\Gamma(\frac{\mu+3}{4})}\right),
\end{align*}
and
\begin{align*}
    \beta_m(\mu)=\frac{i}{2}c_m(\mu-2)\left(\omega_m^1\frac{\Gamma(\frac{\mu+3}{4})}{\Gamma(\frac{\mu+1}{4})}+\omega_m^0\frac{\Gamma(\frac{\mu+1}{4})}{\Gamma(\frac{\mu+3}{4})}\right).
\end{align*}
Furthermore if $\mu\in 1-2\NN$ then $\alpha_m(\mu)$ is only non-zero when $\mu\in D_0$ and $\beta_m(\mu)$ is only non-zero when $\mu \in D_1$. 
\end{lemma}
\begin{proof}
The first identity is a direct consequence of Proposition \ref{P image}. To see the second assertion rewrite
\[
\alpha_m(\mu)=\frac{2^{1-\mu}\pi e^{\frac{\pi i m}{4}}}{\Gamma\big(\frac{\mu+3+|m|}{4}\big)}\bigg(\omega_m^0\frac{\big(\frac{\mu+3-|m|}{4}\big)_{\frac{|m|}{4}}}{\Gamma\big(\frac{\mu+1}{4}\big)}+\omega_m^1\frac{\big(\frac{\mu+3-|m|}{4}\big)_{\frac{|m|-2}{4}}}{\Gamma\big(\frac{\mu+3}{4}\big)} \bigg).
\]
As either $\omega_m^0$ or $\omega_m^1$ is vanishing this makes sense term by term. When $\mu$ is of the form $\mu=4k+3-|m|$ for $k\in \ZZ$ then term by term $\Gamma(\frac{\mu+1}{4})^{-1}$ and $\Gamma(\frac{\mu+3}{4})^{-1}$ vanishes. When $\mu$ has the form $\mu=4k+1-|m|$ for $k\in -\NN$ then $\Gamma(\frac{\mu+3+|m|}{4})^{-1}$ vanishes. A similarly argument applies to $\beta_m(\mu)$.
\end{proof}
\begin{lemma} \label{diskret A billede}
We have 
\[
\mathbb{A}_{\mu}f(k_\theta)=\tilde{\alpha}_m(\mu)J_0 f(\mu,\theta)+\tilde{\beta}_m(\mu)J_1 f(\mu,\theta),
\]
where 
\[
\tilde{\alpha}_m(\mu)=\frac{1}{2}c_{-m}(-\mu)\left(\overline{\omega}_{m}^0\frac{\Gamma(\frac{3-\mu}{4})}{\Gamma(\frac{1-\mu}{4})}+\overline{\omega}_m^1\frac{\Gamma(\frac{1-\mu}{4})}{\Gamma(\frac{3-\mu}{4})} \right), 
\]
\[
\tilde{\beta}_m(\mu)=\frac{i}{4}c_{-m}(-\mu-2)\left(\overline{\omega}_{m}^1\frac{\Gamma(\frac{-\mu+3}{4})}{\Gamma(\frac{-\mu+1}{4})}+\overline{\omega}_m^0\frac{\Gamma(\frac{-\mu+1}{4})}{\Gamma(\frac{-\mu+3}{4})} \right). 
\]
Furthermore, if $\mu\in 1-2\NN$ of the form $\mu=4k+1-|m|$ we have $\tilde{\beta}(\mu)=0$ and similarly for $\mu$ of the form $\mu=4k+3-|m|$ we have $\tilde{\alpha}_m(\mu)=0.$
\end{lemma}
\begin{proof}
This follows from Proposition \ref{P image} and considerations similar to those in the proof of Lemma \ref{Lemma 1 - disk}.
\end{proof}

\begin{lemma}\label{lemma 1}
For $\mu\in D_0$:
\[
\mathbb{P}_{\mu}\mathbb{A}_{\mu} f(k_\theta b_u) = \cosh^{\frac{m}{2}}(2u)\alpha_m(\mu)\widetilde{\alpha}_m(\mu)J_0 f(\mu,\theta)\varphi^m_{\mu}(2u),
\]
and for $\mu \in D_1$:
\begin{align*}
    \mathbb{P}_{\mu}\mathbb{A}_{\mu} f(k_\theta b_u) = \cosh^{\frac{m}{2}}(2u)   \beta_m(\mu)\tilde{\beta}_m(\mu)J_1 f(\mu,\theta)\psi^m_{\mu}(2u).
\end{align*}
Furthermore, if $\mu\in (1-2\NN)\setminus(D_0\cup D_1)$ then $\mathbb{P}_{\mu}\mathbb{A}_{\mu} f(k_\theta b_u)=0$.
\end{lemma}
\begin{proof}
This is a direct consequence of the two preceding lemmas.
\end{proof}
\noindent Consider the non-vanishing entire analytic function
\[
\caprisun(\mu)=\frac{1}{\Gamma\big(\frac{1+\mu}{4}\big)^2\Gamma\big(\frac{1-\mu}{4}\big)^2}+\frac{1}{\Gamma\big(\frac{3+\mu}{4}\big)^2\Gamma\big(\frac{3-\mu}{4}\big)^2}.  
\]
\begin{lemma}\label{lemma 2}
For $\mu\in D_0$ 
\[
(-2\pi)\alpha_m(\mu)\widetilde{\alpha}_m(\mu) \ell_0(\mu)\ell_0(-\mu)
=16\pi^2 \frac{\cot\big(\frac{\pi\mu}{2}\big)\caprisun(\mu)}{\mu}.
\]
For $\mu\in D_1$
\[
\frac{\pi}{2}\beta_m(\mu)\widetilde{\beta}_m(\mu) \ell_1(\mu)\ell_1(-\mu)
=16\pi^2 \frac{\cot\big(\frac{\pi\mu}{2}\big)\caprisun(\mu)}{\mu}.
\]

\end{lemma}
\begin{proof}
This follows from (\ref{c og l}). One trick is used which arises when a term like 
$$\frac{|\omega_m^0|^2}{\Gamma\big(\frac{\mu+1}{2}\big)^2}, $$
is obtained. As $\omega_m^0$ is either 0 or 2 we can set $\omega_m^0=2$ as $\Gamma(\frac{\mu+1}{2})^{-1}$ vanishes in the same cases as $\omega_m^0$.
\end{proof}
\begin{proposition}\label{epsilon =0 disk}
We have
\begin{align*}
    \frac{1}{32\pi}\sum_{\mu\in 1-2\NN} \!\frac{-\mu}{\caprisun(\mu)}\mathbb{P}_\mu \mathbb{A}_\mu f(k_\theta b_u)=
    \cosh^{\frac{m}{2}}(2u)
    \bigg[ 
    &\frac{-1}{2\pi}\sum_{\mu\in D_0}J_0 f(\mu,\theta)\varphi_\mu^m(2u)\underset{\nu=\mu}{\Res}\big(\ell_0(\nu)\ell_0(-\nu)\big)^{-1}
    \\
    &+\frac{2}{\pi}\sum_{\mu\in D_1}J_1 f(\mu,\theta)\psi_\mu^m(2u)\underset{\nu=\mu}{\Res}\big(\ell_1(\nu)\ell_1(-\nu)\big)^{-1}
    \bigg].
\end{align*}
\end{proposition}
\begin{proof}
Apply Lemma \ref{lemma 1} to the right hand side. Now note that $c_m(\mu)c_{-m}(-\mu)$ is regular for $\mu \in D_0$ and $c_m(\mu-2)c_{-m}(-\mu-2)$ is regular for $\mu\in D_1$, thus they can be moved inside the residues. Then everything follows from Lemma \ref{lemma 2} after recalling that $\Res_{\nu=\mu}\tan(\frac{\pi\nu}{2})=-\frac{\pi}{2}$.
\end{proof}

\subsection{The discrete part for \texorpdfstring{$\varepsilon=1$}{epsilon=1}} 
\noindent In this subsection we fix $\varepsilon=1$. The proof will proceed using the same ideas as for $\varepsilon=0$. For $\mu\in -2\NN$ let 
\[
\mathcal{A}_\mu=\frac{1}{\Gamma\big(\frac{1-\mu}{2}\big)}A_\mu^0,\quad\text{and}\quad \mathcal{P}_\mu\zeta_m=\frac{(-1)^{\frac{m+|m|-2}{2}}}{\Gamma\big(\frac{1+\mu}{2}\big)}P_\mu^1\zeta_m,
\]
that is we define $\mathcal{P}_\mu$ by its eigenvalues on $K$-types. By Proposition \ref{P image} we get $\mathcal{P}_\mu$ is intertwining by the same argument we used for $\hat{T}^1_\mu$ in section \ref{Principal series section}. 
\begin{lemma}
For $\mu\in D_0$ we have
\[
\mathcal{P}_\mu\mathcal{A}_\mu f(k_\theta b_u)=\alpha_m(\mu)\cosh^{\frac{m}{2}}(2u)\varphi_\mu^m(2u)J_0 f(\mu,\theta), 
\]
where 
\[
\alpha_m(\mu)=i(-1)^\frac{|m|+1}{2}c_m(\mu)c_{-m}(-\mu)
\]
For $\mu\in D_1$ we have
\[
\mathcal{P}_\mu\mathcal{A}_\mu f(k_\theta b_u)=\beta_m(\mu)\cosh^{\frac{m}{2}}(2u)\psi_\mu^m(2u)J_1 f(\mu,\theta),
\]
where 
\[
\beta_m(\mu)=\frac{1}{4} i(-1)^\frac{|m|+1}{2}c_m(\mu-2)c_{-m}(-\mu-2). 
\]
Furthermore if $\mu\in -2\NN$ then $\alpha_m(\mu)$ is only non-zero if $\mu \in D_0$ and $\beta_m(\mu)$ is only non-zero if $\mu \in D_1$.
\end{lemma} 
\begin{proof}
The proof is an application of Propositions \ref{P image} and \ref{A image}.
\end{proof}
\begin{lemma}
For $\mu \in D_0$
\[
(-2\pi)\alpha_m(\mu)\ell_0(\mu) \ell_0(-\mu)=4i(-1)^\frac{|m|-1}{2}\frac{\sin\big(\frac{\pi\mu}{2}\big)}{\mu},
\]
and for $\mu\in D_1$
\[
\frac{\pi}{2}\beta_m(\mu)\ell_1(\mu) \ell_1(-\mu)=4i(-1)^\frac{|m|+1}{2}\frac{\sin\big(\frac{\pi\mu}{2}\big)}{\mu}.
\]
\end{lemma}
\begin{proof}
This is a direct consequence of (\ref{c og l}).
\end{proof}
\begin{proposition}
\begin{align*}
    \frac{1}{2\pi i}\sum_{\mu\in -2\NN} \mu\mathcal{P}_\mu \mathcal{A}_\mu f(k_\theta b_u)=
    \cosh^{\frac{m}{2}}(2u)
    \bigg[ 
    &-\frac{1}{2\pi}\sum_{\mu\in D_0}J_0 f(\mu,\theta)\varphi_\mu^m(2u)\underset{\nu=\mu}{\Res}\big(\ell_0(\nu)\ell_0(-\nu)\big)^{-1}
    \\
    &+\frac{2}{\pi}\sum_{\mu\in D_1}J_1 f(\mu,\theta)\psi_\mu^m(2u)\underset{\nu=\mu}{\Res}\big(\ell_1(\nu)\ell_1(-\nu)\big)^{-1}
    \bigg].
\end{align*}
\end{proposition}
\begin{proof}
This follows in the same manner as the proof for Proposition \ref{epsilon =0 disk}, where we here note for $\mu=4k+1-|m|\in D_0$ that $\Res_{\nu=\mu}\sin\big(\frac{\pi \nu }{2}\big)^{-1}=\frac{2}{\pi}(-1)^\frac{|m|-1}{2}, $ and for $\mu\in D_1$ we have $\Res_{\nu=\mu}\sin\big(\frac{\pi \nu }{2}\big)^{-1}=\frac{2}{\pi}(-1)^\frac{|m|+1}{2} $.
\end{proof}

\section{The Plancherel formula}
\noindent The intertwining operators $P_\mu^\xi$ and $A_\mu^\xi$ are continuous maps and hence the intertwining operators introduced in the previous section are also continuous. This allows for an extension of the results obtained for K-types, described by the first theorem of this section. We then extend this theorem to arbitrary $\lambda\in i\RR$ by virtue of Theorem \ref{indre Ind}

Recall that 
\begin{align} \label{a caprisun}
a(\mu)= 2^{\frac{3}{2}}\pi\frac{\Gamma(\frac{\mu}{2})}{\Gamma(\frac{1+\mu+\varepsilon}{2})\Gamma(\frac{1+\mu-\varepsilon}{2})},\quad \text{and}\quad \caprisun(\mu)=\frac{1}{\Gamma\big(\frac{1+\mu}{4}\big)^2\Gamma\big(\frac{1-\mu}{4}\big)^2}+\frac{1}{\Gamma\big(\frac{3+\mu}{4}\big)^2\Gamma\big(\frac{3-\mu}{4}\big)^2}. 
\end{align}
\begin{theorem}[Plancherel formula for $\lambda=0$] \label{Plancherel lambda=0}
For $\varepsilon=0$ and $f\in C_c^\infty\mbox{--}\Ind_H^G(\varepsilon\otimes e^0)$ we have the following inversion formula 
\[
f(k_\theta b_u)= \int_{i\RR} \sum\limits_{\xi=0}^1\textbf{P}_{\mu}^\xi\textbf{A}_{\mu}^\xi f(k_\theta b_u) \frac{d\mu}{|a(\mu)|^2}+\frac{1}{32\pi}\sum_{\mu \in 1-2\NN}\frac{-\mu}{\caprisun(\mu)} \mathbb{P}_{\mu}\mathbb{A}_{\mu} f(k_\theta b_u), 
\]
and the corresponding Plancherel formula
\begin{align*}
    \| f\|^2
    &=\int_{i\RR} \sum_{\xi=0}^1\|\textbf{A}^\xi_\mu f \|^2 \,\frac{d\mu}{|a(\mu)|^2}+\frac{1}{16\pi}\sum_{\mu \in 1-2\NN}\frac{\Gamma(1-\mu)}{\Gamma\big(\frac{-\mu}{2}\big)\caprisun(\mu)}\|\mathbb{A}_\mu f\|^2.
\end{align*}
For $\varepsilon=1$ and $f\in C_c^\infty\mbox{--}\Ind_H^G(\varepsilon\otimes e^0)$ we have the following inversion formula 
$$f(k_\theta b_u)= \int_{i\RR} \sum\limits_{\xi=0}^1\textbf{P}_{\mu}^\xi\textbf{A}_{\mu}^\xi f(k_\theta b_u) \frac{d\mu}{|a(\mu)|^2}+\frac{1}{2\pi i}\sum_{\mu \in -2\NN}\mu \mathcal{P}_{\mu}\mathcal{A}_{\mu} f(k_\theta b_u), $$
and the corresponding Plancherel formula
\begin{align*}
    \| f\|^2
    =\int_{i\RR} \sum_{\xi=0}^1\|\textbf{A}^\xi_\mu f \|^2 \,\frac{d\mu}{|a(\mu)|^2}+\frac{1}{2\pi}\sum_{\mu \in -2\NN}\frac{\Gamma(-\mu )\mu^2}{\Gamma\big(\frac{1-\mu}{2}\big)}\|\mathcal{A}_\mu f\|^2.
\end{align*}
\end{theorem}
\begin{proof}
The inversion formulas follow directly from the introduction and results of Section \ref{composing}. To get the Plancherel formula write 
\[
\|f \|^2=\int_0^\pi \int_\RR f(k_\theta b_u)\overline{f(k_\theta b_u)}\cosh(2u)\,du d\theta,
\]
and use the inversion formula on $f(k_\theta b_u)$ and apply that 
\[
\int_{G/H} P^\xi_\mu f(xH) \overline{g(xH)}\,d(xH)=\int_K f(k)\overline{A_{-\overline{\mu}}^\xi g(k)}\,dk,
\]
for $f\in \Ind_P^G(\varepsilon\otimes e^\mu)$ and $g\in \Ind_H^G(\varepsilon\otimes e^0)$. Lastly for the discrete part, we apply Proposition \ref{Tony transform} to get 
\[
T_\mu^0\mathbb{A}_\mu=\frac{\sqrt{\pi}2^\mu}{\Gamma\big(\frac{1-\mu}{2}\big)}\mathbb{A}_{-\mu}\quad \text{and}\quad T_\mu^1\mathcal{A}_\mu=\frac{\sqrt{\pi}2^\mu}{\Gamma\big(\frac{2-\mu}{2}\big)\Gamma\big(\frac{1+\mu}{2}\big)}A_{-\mu}^1,
\]
giving the final result.
\end{proof}
\noindent We now extend the previous result from $\lambda=0$ to $\lambda\in i\RR$ using Theorem \ref{indre Ind}. We want to compose $A_{\mu}$ and $\ankh_\lambda^0$ but as we cannot ensure the regularity of the functions in the image of $\ankh_\lambda^0$ we end up doing this in an $L^2$-sense using direct integrals. Consider the following operators 
\[
\mathbb{A}_{\lambda,\mu}:=\frac{2^{\frac{1+\mu}{2}}\sqrt{\pi}\Gamma\big(\frac{1+\mu}{4}+\frac{\lambda}{4}\big)}{\Gamma\big(\frac{1-\mu}{4}\big)\Gamma\big(\frac{1+\mu}{4}\big)\Gamma\big(\frac{1-\mu}{4}-\frac{\lambda}{4}\big)}A_{\lambda,\mu}^0+\frac{2^{\frac{1+\mu}{2}}\sqrt{\pi}\Gamma\big(\frac{3+\mu}{4}+\frac{\lambda}{4}\big)}{\Gamma\big(\frac{3-\mu}{4}\big)\Gamma\big(\frac{3+\mu}{4}\big)\Gamma\big(\frac{3-\mu}{4}-\frac{\lambda}{4}\big)}A_{\lambda,\mu}^1,
\]
\[
\mathbf{A}_{\lambda,\mu}^\xi:= \frac{A_{\lambda,\mu}^\xi}{\Gamma(\frac{1-\mu}{2})},\quad \text{and}\quad \mathcal{A}_{\lambda,\mu}:=\frac{2^{\frac{1+\mu}{2}}\sqrt{\pi}\Gamma\big(\frac{1+\mu}{4}+\frac{\lambda}{4}\big)}{\Gamma\big(\frac{1-\mu}{4}\big)\Gamma\big(\frac{3+\mu}{4}\big)\Gamma\big(\frac{1-\mu}{4}-\frac{\lambda}{4}\big)}A_{\lambda,\mu}^0, 
\]
which are extensions of $\mathbb{A}_\mu$, $\textbf{A}_\mu^\xi$ and  $\mathcal{A}_\mu$ e.g. $\mathbb{A}_{0,\mu}=\mathbb{A}_\mu$. Furthermore let 
\begin{align*}
    \calH_\varepsilon = \int_{i\RR}^\oplus \pi_{\varepsilon,\mu} \otimes \CC^2\, d\mu \oplus \bigoplus\limits_{\mu \in 1-\varepsilon-2\NN }\pi^{\text{ds}}_{\varepsilon,\mu}.
\end{align*}
where $\int^\oplus_U H_\mu d\mu$ denotes a direct integral of Hilbert spaces, see e.g \cite{JF22} for a short exposition. The inner-product on $\calH_\varepsilon$ is given by Theorem \ref{Plancherel lambda=0}, i.e. for $\varepsilon=0$ and $f,h\in \calH_0$ 
\begin{align*}
    \langle g,h \rangle_\calH = \int_{i\RR} \sum\limits_{\xi=0}^1\langle g^\xi_\mu,h_\mu^\xi \rangle_{L^2(K)} \frac{d\mu}{|a(\mu)|^2} + \frac{1}{16\pi}\sum\limits_{\mu \in 1-2\NN} \frac{\Gamma(1-\mu)}{\Gamma(\frac{-\mu}{2})\caprisun(\mu)}\langle T^0_\mu g_{\mu},h\rangle_{L^2(K)}.
\end{align*}
For simplicity we shall assume that $\varepsilon=0$ for the remainder of the section. All arguments made can be done for $\varepsilon=1$ as well using the corresponding results from the previous section.

Abusing notation, Theorem \ref{Plancherel lambda=0} defines an isometry $A_0:C_c^\infty(G/H)\rightarrow \calH_0$ which extends to an isometry
\[A_0:\Ind_H^G(\varepsilon \otimes 1) \rightarrow \calH_0.\]
For $f\in \calH_0$ with $f=(f^0,f^1,f^d)$ we introduce the following map
\begin{align*}
    &P_0:\calH_0 \rightarrow L^2(G/H)\\
    f=(f^0,f^1,f^d) \mapsto \int_{i\RR} &\sum\limits_{\xi=0}^1\textbf{P}^\xi_\mu f^\xi_\mu \frac{d\mu}{|a(\mu)|^2} + \frac{1}{32\pi}\sum\limits_{\mu \in 1-2\NN}\frac{-\mu}{\caprisun(\mu)}\mathbb{P}_\mu f^d_\mu.
\end{align*}

\begin{lemma}\label{0'rne var de fede tider}
For $f\in \Ind_H^G(\varepsilon \otimes e^0)$ and $h\in \calH_0$ we have the following relation:
\begin{align*}
    \langle A_0 f, h\rangle_\calH = \langle f, P_0h\rangle_{L^2(G/H)}.
\end{align*}
\end{lemma}

\begin{proof}
Let $f\in \Ind_H^G(\varepsilon \otimes e^\lambda)$ and $h\in C^\infty_c(\calH_0)$. We then find
\begin{align*}
    &\langle A_0 f,h\rangle_\calH= \sum\limits_{\xi=0}^1\int_{i\RR}\int_K \textbf{A}_\mu^\xi f(k)\overline{h^\xi_\mu(k)} dk\frac{d\mu}{|a(\mu)|^2}+ \frac{1}{16\pi}\sum\limits_{\mu \in 1-2\NN} \frac{\Gamma(1-\mu)}{\Gamma(\frac{-\mu}{2})\caprisun(\mu)} \langle T^0_\mu \circ \mathbb{A}_{\mu} f,h^\xi_\mu\rangle_{L^2(K)}\\
    &=\sum\limits_{\xi=0}^1 \int_{i\RR} \int_{G/H} f(x) \overline{\textbf{P}_\mu^\xi h^\xi_\mu(x)} d(xH) + \frac{1}{16\pi}\sum\limits_{\mu \in 1-2\NN}  \frac{\Gamma(1-\mu)}{\Gamma(\frac{-\mu}{2})\caprisun(\mu)} \frac{\sqrt{\pi}2^\mu}{\Gamma\big(\frac{1-\mu}{2}\big)}\langle \mathbb{A}_{-\mu} f, h^d_\mu \rangle_{L^2(K)}\\
    &=\sum\limits_{\xi=0}^1 \int_{i\RR} \langle f, \textbf{P}_\mu^\xi h^\xi_\mu\rangle_{L^2(G/H)}\frac{d\mu}{|a(\mu)|^2} + \frac{1}{32\pi}\sum\limits_{\mu \in 1-2\NN}  \frac{-\mu}{\caprisun(\mu)}\langle f, \mathbb{P}_{\mu}h^d_\mu \rangle_{L^2(G/H)}\\
    &=\langle f, P_0 h\rangle_{L^2(G/H)}. \qedhere
\end{align*}
\end{proof}
\begin{lemma}[See {{\cite[Theorem 1]{PEN76}}}]\label{Magisk direkte int lemma}
Suppose $S:\calH^\infty\rightarrow \calH^\infty$ is a continuous intertwining operator for $\calH^\infty$. Then for a.e $\mu \in i\RR\cup (1-\varepsilon-2\NN)$ there exists unique $\calH^\infty$ intertwining operators $S_\mu$ for $\pi_{\varepsilon,\mu}^\infty\otimes \CC^2$ if $\mu \in i\RR$ and for $\pi_{\varepsilon,\mu}^{\text{ds}}$ if $\mu \in 1-\varepsilon-2\NN$ such that
\[
(Sf)_\mu = S_\mu f_\mu \quad \text{a.e } \mu\in i\RR\cup (\mu\in 1-\varepsilon-2\NN),\,f\in \calH^\infty .
\]
\end{lemma}
\begin{proposition}\label{A er bare på}
The map $A^\varepsilon_0:\Ind_H^G(\varepsilon\otimes e^0)\rightarrow \calH_0$ is surjective. In particular $A_0$ is an isometric isomorphism. 
\end{proposition}
\begin{proof}
Since the discrete and continuous part of $\calH_0$ consist of pairwise inequivalent representations of $G$ it suffices to show that the projection of $W=\overline{\text{Image}(A_0)}$ onto the continuous part and the discrete part respectively, is surjective. For the projection to the discrete part we can consider the projection of $W$ onto each summand $\pi^{\text{ds}}_{0,\mu}$. Proposition \ref{Best of casselmand} gives $\pi_{0,\mu}^{\text{ds}} = \pi^{\hol}_{0,\mu}\oplus \pi^{\ahol}_{0,\mu}$ and since these representations are inequivalent it again suffices to show that the projection on each of them are onto. Lemma \ref{diskret A billede} then shows that $\proj_{\pi_{0,\mu}^\hol}(W) \neq 0$ and $\proj_{\pi^{\ahol}_{0,\mu}}(W)\neq 0$. But since the projection is $G$-equivariant the image is a subrepresentation and it follows that both projections must be onto.

Since the projection onto an integrand of the continuous part of $\calH_0$ is in general not point-wise defined, the proof differs to that of the discrete part. Abusing notation slightly we shall write $f_\mu=(f_\mu^0,f_\mu^1)$ when $\mu\in i\RR$ and $f\in \calH_0$, omitting the discrete part.

By Lemma \ref{0'rne var de fede tider}  $(A_0)^*=P_0$ and since the adjoint is $G$-equivariant we have 
\[P_0(\calH_0^\infty)\subseteq L^2(G/H)^\infty\quad \text{ and }\quad A_0(L^2(G/H)^\infty)\subseteq \calH_0^\infty.\] 
Let $A_0^\infty=A_0|_{L^2(G/H)^\infty}$ and $P^\infty_0=P_0|_{\calH_0^\infty}$. Then 
\[
S=A^\infty_0\circ P^\infty_0:\calH_0^\infty\rightarrow \calH_0^\infty
\]
is a $\calH_0^\infty$ intertwining map and by Lemma \ref{Magisk direkte int lemma} there exists a family of $\calH_0^\infty$ intertwining operators $(S_\mu)$ such that $((A^\infty_0\circ P^\infty_0)f)_\mu = S_\mu f_\mu$ for a.e $\mu \in i\RR$ and all $f\in \calH_0^\infty$. By Schur's lemma this implies $S_\mu = (\id\otimes B_\mu)$ for a.e $\mu \in i\RR$, with $B_\mu:\CC^2\rightarrow \CC^2$ a linear map. Let $N$ denote the corresponding null-set, we then show that for $\mu\in i\RR\setminus N$ we have $S_\mu=\id$.

To this extent let $\mu \in i\RR\setminus N$ and let $f$ be a $K$-finite vector in $\calH_0^\infty$ and note that this implies $f_\mu$ must be a $K$-finite vector in $\pi_\mu^\infty$. Assume therefore without loss of generality that $f_\mu = (c_1\zeta_m,c_2\zeta_n)$ for some $m,n\in 2\ZZ$ and pick by Proposition \ref{A image} a $K$-finite vector in $L^2(G/H)^\infty$ such that $A_0(w)_\mu = f_\mu$. One can e.g pick $w$ of the form $w=\zeta_m f_1 + \zeta_n f_2$ with $f_1$ an even function and $f_2$ an odd function of correct regularity and growth. Then we have
\begin{align*}
(A^\infty_0\circ P^\infty_0 f)_\mu = A^\infty_0\circ P^\infty_0 \circ A^\infty_0(w)_\mu = A_0^\infty(w)_\mu = (A_\mu^0 w, A_\mu^1 w),
\end{align*}
where the second equality follows from the inversion formula given by Theorem \ref{Plancherel lambda=0}. On the other hand we have
\begin{align*}
    (A^\infty_0\circ P^\infty_0 f)_\mu=S_\mu f_\mu =  (\id\otimes B_\mu)A_0^\infty(w)_\mu = ((b_{11})_\mu A_\mu^0 w + (b_{12})_\mu A^1_\mu w,(b_{21})_\mu A_\mu^0 w + (b_{22})_\mu A^1_\mu w)
\end{align*}
hence
\begin{align*}
(A_\mu^0 w, A_\mu^1 w) = (a_\mu A_\mu^0 w + b_\mu A^1_\mu w,c_\mu A_\mu^0 w + d_\mu A^1_\mu w).
\end{align*}
Since $A_\mu^0$ and $A_\mu^1$ are linearly independent for $\mu\in i\RR$ it follows that $B_\mu = \id$ and hence $S_\mu = \id$ on the $K$-finite vectors of $\calH_0$, for a.e $\mu \in i\RR$ and since the $K$-finite vectors form a dense subset the result follows.
\end{proof}

\begin{theorem}
For $\lambda\in i\RR$ and $\varepsilon\in \ZZ/2\ZZ$ we have
\[
\Ind_H^G(\varepsilon\otimes e^\lambda)\cong \int_{i\RR}^{\oplus} \pi_{\varepsilon,\mu}\oplus \pi_{\varepsilon,\mu}\,\frac{d\mu}{|a(\mu)|^2}\,\oplus\bigoplus_{\mu\in 1-\varepsilon-2\NN} \pi_{\varepsilon,\mu}^{\hol}\oplus \pi_{\varepsilon,\mu}^\ahol,
\]
where the map is given by
\begin{align*}
f \mapsto (p_{\lambda,\mu}^0\textbf{A}_{\lambda,\mu}^0f,p_{\lambda,\mu}^1\textbf{A}_{\lambda,\mu}^1f,\mathbb{A}_{\lambda,\mu} f), \qquad \text{for } f\in \Ind_{H}^G(0\otimes e^\lambda)\\
f \mapsto (p_{\lambda,\mu}^0\textbf{A}_{\lambda,\mu}^0f,p_{\lambda,\mu}^1\textbf{A}_{\lambda,\mu}^1f,\mathcal{A}_{\lambda,\mu} f), \qquad \text{for } f\in \Ind_{H}^G(1\otimes e^\lambda)
\end{align*}
\end{theorem}
\begin{proof}
For $\lambda\in i\RR$ the map $A_0:\Ind_H^G(\varepsilon \otimes 1) \rightarrow \calH_0$ gives rise to an isometric isomorphism $\Ind_H^G(\varepsilon\otimes e^\lambda)\rightarrow \calH_0$ by composition with $\ankh_\lambda^0:\Ind_H^G(\varepsilon\otimes e^\lambda) \rightarrow \Ind_H^G(\varepsilon\otimes e^0)$ from Theorem \ref{indre Ind}. Let $\lambda \in i\RR$, $f\in \Ind_H^G(\varepsilon \otimes e^\lambda)$ and $h\in C_c^\infty(\calH_0)$. Then by Lemma \ref{0'rne var de fede tider} we have
\begin{align*}
    &\langle A_0 \circ \ankh_\lambda^0 f,h\rangle_\calH = \langle \ankh_\lambda^0 f, P_0 h\rangle_{L^2(G/H)}\\
    &=\sum\limits_{\xi=0}^1 \int_{i\RR} \langle \ankh_\lambda^0f, \textbf{P}_\mu^\xi h^\xi_\mu\rangle_{L^2(G/H)}\frac{d\mu}{|a(\mu)|^2} + \frac{1}{32\pi}\sum\limits_{\mu \in 1-2\NN}  \frac{-\mu}{\caprisun(\mu)}\langle \ankh_\lambda^0 f, \mathbb{P}_{\mu}h^d_\mu \rangle_{L^2(G/H)}\\
    &=\sum\limits_{\xi=0}^1\int_{i\RR}\int_K\int_{G/H} \ankh_\lambda^0f(x) \frac{K^\xi_{-\mu}(x^{-1}k)}{\Gamma(\frac{1-\mu}{2})}\overline{h_\mu^\xi(k)} d(xH)dk \frac{d\mu}{|a(\mu)|^2}\\ &
    \quad +\frac{1}{32\pi}\sum\limits_{\mu \in 1-2\NN}  \frac{-\mu}{\caprisun(\mu)}\langle \ankh_\lambda^0 f, \mathbb{P}_{\mu}h^d_\mu \rangle_{L^2(G/H)}.
\end{align*}
Using the coordinates $xH=k_\theta\overline{n}_yH$ and applying Lemma \ref{convolution constant} in the distributional sense, we find
\begin{align*}
    &\frac{1}{\Gamma(\frac{1-\mu}{2})}\int_{G/H} K_{-\mu}^\xi(x^{-1}k)\ankh_\lambda^0f(x) d(xH)\\
    &=\frac{\Gamma(\frac{2-\lambda}{4})}{\sqrt{\pi}2^{\frac{\lambda}{4}}\Gamma(\frac{1-\mu}{2})\Gamma(\frac{\lambda}{4})}\int_0^\pi \!\int_\RR |\cos \theta |_\varepsilon^{-\mu-1} f(gk_\theta \overline{n}_z)\int_\RR |x|^{\frac{-\lambda-2}{2}}_\varepsilon |z-\tan \theta -x|^{\frac{-\mu-1}{2}}_\xi \frac{1}{2}dxdzd\theta \\
    &=p_{\lambda,\mu}^\xi \textbf{A}_{\lambda,\mu}^{\xi}f(k).
\end{align*}
An analogous calculation applies to the discrete part after applying that $\mathbb{A}_{-\mu}=\frac{\Gamma\big(\frac{1-\mu}{2}\big)}{\sqrt{\pi}2^\mu}T_\mu^0\mathbb{A}_\mu$. In conclusion we find
\begin{align*}
   \langle A_0 \circ \ankh_\lambda^0 f,h\rangle_\calH = \sum\limits_{\xi=0}^1 \int_{i\RR} \langle p_{\lambda,\mu}^\xi\textbf{A}^\xi_{\lambda,\mu} f, h^\xi_\mu\rangle_{L^2(K)} \frac{d\mu}{|a(\mu)|^2} + \frac{1}{16\pi}\sum\limits_{\mu \in 1-2\NN}  \frac{\Gamma(1-\mu)}{\Gamma(\frac{-\mu}{2})\caprisun(\mu)}\langle T_\mu^0\mathbb{A}_{\lambda,\mu} f, h^d_\mu \rangle_{L^2(K)}
\end{align*}
Hence $A_0\circ \ankh_\lambda^0 = A_{\lambda}$ with $A_{\lambda}:\Ind_H^G(\lambda\otimes e^\lambda) \rightarrow \calH_0$ given by
\begin{align*}
    \langle A_\lambda f, h\rangle_\calH = \sum\limits_{\xi=0}^1 \int_{i\RR} \langle p_{\lambda,\mu}^\xi \textbf{A}^\xi_{\lambda,\mu} f, h^\xi_\mu\rangle_{L^2(K)} \frac{d\mu}{|a(\mu)|^2} + \frac{1}{16\pi}\sum\limits_{\mu \in 1-2\NN}  \frac{\Gamma(1-\mu)}{\Gamma(\frac{-\mu}{2})\caprisun(\mu)}\langle T_\mu^0\mathbb{A}_{\lambda,\mu} f, h^d_\mu \rangle_{L^2(K)}\label{Ankhumentet}
    \end{align*}
\end{proof}
\begin{corollary}\label{hovedresultat}
For $\varepsilon=0$ and $\Ind_H^G(\varepsilon\otimes e^\lambda)$ we have the following Plancherel formula 
\begin{align*}
    \| f\|^2
    &=\int_{i\RR} \sum_{\xi=0}^1\|\textbf{A}^\xi_{\lambda,\mu} f \|^2 \,\frac{d\mu}{|a(\mu)|^2}+\frac{1}{16\pi}\sum_{\mu \in 1-2\NN}\frac{\Gamma(1-\mu)}{\Gamma\big(\frac{-\mu}{2}\big)\caprisun(\mu)}\|\mathbb{A}_{\lambda,\mu} f\|^2.
\end{align*}
For $\varepsilon=1$ and $f\in \Ind_H^G(\varepsilon\otimes e^\lambda)$ we have the following Plancherel formula 
\begin{align*}
    \| f\|^2
    =\int_{i\RR} \sum_{\xi=0}^1\|\textbf{A}^\xi_{\lambda,\mu} f \|^2 \,\frac{d\mu}{|a(\mu)|^2}+\frac{1}{2\pi}\sum_{\mu \in -2\NN}\frac{\mu^2\Gamma(-\mu)}{\Gamma\big(\frac{1-\mu}{2}\big)}\|\mathcal{A}_{\lambda,\mu} f\|^2.
\end{align*}
with $a(\mu)$ and $\caprisun(\mu)$ given by \eqref{a caprisun}.
\end{corollary}

\begin{proof}
Since $|p_{\lambda,\mu}^\xi|=1$ for $\lambda,\mu \in i\RR$ the assertion follows.
\end{proof}

\appendix
\section{Integral formulas}
\begin{lemma}[See {{\cite[Proposition 16.8]{C20}}}] \label{casselman integral}
For $\Re{(\alpha+\beta)}>0$ we have 
\[
\int_\RR \frac{dx}{(x-i)^\alpha(x+i)^\beta}=\frac{2^{2-\alpha-\beta} i^{\alpha-\beta} \Gamma(\alpha+\beta-1)}{\Gamma(\alpha)\Gamma(\beta)}
\]
\end{lemma}
\begin{lemma}[See {{\cite[Section 3.631]{GR94}}}]
For $\Re{\nu}>0$ we have
\[
\int_0^\pi \sin^{\nu-1}(x)\cos(ax)\,dx=\frac{2^{1-\nu}\pi \cos\big(\frac{a\pi}{2}\big)\Gamma(\nu)}{\Gamma\big(\frac{\nu+1-a}{2}\big)\Gamma\big(\frac{\nu+1+a}{2}\big)}.
\]
\end{lemma}
\begin{lemma} [See {{\cite[Section 3.631]{GR94}}}] 
For $\Re{\nu}>0$ we have
\[
\int_0^\pi \sin^{\nu-1}(x)\sin(ax)\,dx=\frac{2^{1-\nu}\pi \sin\big(\frac{a\pi}{2}\big)\Gamma(\nu)}{\Gamma\big(\frac{\nu+1-a}{2}\big)\Gamma\big(\frac{\nu+1+a}{2}\big)}.
\]
\end{lemma}
\begin{lemma} \label{appendix trig integral}
For $\Re{\nu}>0$ we have 
\[
\int_0^\pi \sin^{\nu-1}(x) e^{iax}\,dx=\frac{2^{1-\nu}\pi e^{\frac{ai\pi}{2}}\Gamma(\nu)}{\Gamma\big(\frac{\nu+1-a}{2}\big)\Gamma\big(\frac{\nu+1+a}{2}\big)}.
\]
\end{lemma}
\begin{lemma}[See {{\cite[Section 3.251]{GR94}}}]\label{1. integral}
For $\Re{\beta}>-1$ and $\Re{(\alpha+\beta)}<-1$ we have
\[
\int_1^\infty x^\alpha(x-1)^\beta \,dx=B\big(-\alpha-\beta-1,\beta+1\big).
\]
\end{lemma}
\begin{lemma}[See {{\cite[Section 3.194]{GR94}}}]\label{2. integral}
For $\Re{\beta}>-1$ and $\Re{(\alpha+\beta)}<-1$ we have
\[
\int_0^\infty x^\alpha(x+1)^\beta \,dx=B\big(-\alpha-\beta-1,\alpha+1\big).
\]
\end{lemma}

\section{The Fourier Transform and Riesz distributions} \label{Fourier transform and Riesz distributions}
\noindent Define the Fourier transform of $\varphi\in C_c(\RR)$ as 
\[
\mathcal{F}[\varphi](\xi)=\int_\RR \varphi(x)e^{ix\xi}\,dx
\]
which makes the inversion formula $\mathcal{F}\mathcal{F}[\varphi](x)=2\pi\varphi(-x)$. Extend this to distributions in the usual way.

For $\alpha\in \CC$ with $\Re{\alpha}>-1$ and $\varepsilon\in \{0,1\}$ the function 
\[
u_\alpha^\varepsilon (x)=\frac{1}{2^\frac{\alpha}{2}\Gamma\big(\frac{\alpha+1+\varepsilon}{2}\big)}|x|^\alpha_\varepsilon,
\]
is locally integrable and can thus be considered as a distribution. 
\begin{lemma}
The family of distributions $u_\alpha^\varepsilon$ extends analytically to a holomorpic family in $\alpha\in \CC$. For $\alpha=1-\varepsilon-2n\in 1-\varepsilon-2\NN$ we have 
\[
u_{1-\varepsilon-2n}^\varepsilon(x)=\frac{(-1)^{n+\varepsilon -1} (n-1)!}{2^{\frac{1-\varepsilon}{2}-n}(2n+\varepsilon-2)!}\delta^{(2n+\varepsilon-2)}(x),
\]
where $\delta(x)$ is the Dirac $\delta$--function.
\end{lemma}
\begin{lemma}\label{convolution constant}
For $\alpha \in \CC$ we have
\[
\mathcal{F}[u_\alpha^\varepsilon]=\sqrt{2\pi}i^\varepsilon u_{-\alpha-1}^\varepsilon.
\]
Furthermore for $\alpha,\beta \in \CC $ with $\Re{\alpha},\Re{\beta}>-1$ and $\Re{(\alpha+\beta)}<-1$ we get 
\[
\int_\RR u_\alpha^\varepsilon(x)u_\beta^\xi(y-x)\,dx =(-1)^{\lfloor \frac{\varepsilon+\xi}{2}\rfloor}\sqrt{2\pi} \frac{\Gamma\big(\frac{-1-\alpha-\beta+[\varepsilon+\xi]_2}{2}\big)}{\Gamma\big(\frac{-\alpha+\varepsilon}{2}\big)\Gamma\big(\frac{-\beta+\xi}{2}\big)}u_{\alpha+\beta+1}^{\varepsilon+\xi}(y),
\]
for $y\neq 0$. 
\end{lemma}
\begin{proof}
The first assertion can be found in \cite[p.170]{GS64}. For $y\neq 0$ we have 
\[
\int_\RR |x|^\alpha_\varepsilon|y-x|^\beta_\xi\,dx =|y|^{\alpha+\beta+1}_{\xi+\varepsilon}\int_\RR |x|^\alpha_\varepsilon|1-x|^\beta_\xi \,dx,
\]
by change of variables. Now writing 
\[
\int_\RR |x|^\alpha_\varepsilon|1-x|^\beta_\xi \,dx=(-1)^\varepsilon\int_0^\infty x^\alpha(1+x)^\beta\,dx
+\int_0^1x^\alpha(1-x)^\beta\,dx
+(-1)^\xi \int_1^\infty x^\alpha(x-1)^\beta\,dx,
\]
we can use the integral formula for the Beta-function, apply Lemma \ref{1. integral}, \ref{2. integral} and then do some Gamma-function yoga.
\end{proof}

\section{Fourier-Jacobi transform} \label{Fourier Jacobi}
\noindent This section is a condensed form of \cite[Appendix 1]{FJ77}. 
For $\alpha,\beta \in \CC$ with $\alpha\notin -\NN $ and $\Re{\beta}>-1$, define the Fourier--Jacobi transform of $f\in C_c^\infty(\RR_{\geq 0})$ by
\[
J_{\alpha,\beta}f(\mu)=\int_0^\infty f(t) \phi^{\alpha,\beta}_\mu(t) \cosh^{2\alpha+1}(t)\sinh^{2\beta+1}(t)\,dt,
\]
where $\phi^{\alpha,\beta}_\mu$ are the Jacobi functions given by
\[
\phi^{\alpha,\beta}_\mu(t)=\, _2F_1\left(\frac{\alpha+\beta+1+\mu}{2},\frac{\alpha-\beta+1+\mu}{2};\alpha+1;-\sinh^2(t)\right).
\]
Then we have the following inversion formula:
\begin{multline*}
    f(t)=\frac{1}{4\pi}\int_{i\RR} J_{\alpha,\beta}(\mu)\phi^{\alpha,\beta}_\mu(t)\frac{d\mu}{|c_{\alpha,\beta}(\mu)|^2}
    -\sum_{\mu\in D_{\alpha,\beta}} J_{\alpha,\beta}(\mu)\phi^{\alpha,\beta}_\mu(t)\underset{\nu=\mu}{\Res}\big(c_{\alpha,\beta}(\nu)c_{\alpha,\beta}(-\nu)\big)^{-1},
\end{multline*}
where
\[
c_{\alpha,\beta}(\mu)=\frac{\Gamma(\mu)\Gamma(\alpha+1)}{\Gamma\big(\frac{\alpha+|\beta|+1+\mu}{2}\big)\Gamma\big(\frac{\alpha-|\beta|+1+\mu}{2}\big)},
\]
and 
\[
D_{\alpha,\beta}=\{x \in \RR \,|\, k\in \NN_0,\, x=2k+1+\alpha-|\beta|<0\}.
\]

\bibliographystyle{amsplain}
\bibliography{biblio.bib}

\end{document}